\documentclass[11	pt,notitlepage]{article}
\usepackage[a4paper,margin=0.96in]{geometry}
\usepackage{xcolor}
\usepackage{booktabs}
\usepackage{graphicx,wrapfig,graphics, amsfonts, amssymb, mathtools, bbm, bm, amsthm, float, dsfont, array, tikz, enumitem}
\usepackage{url}
\usepackage[hidelinks]{hyperref}
\newcommand{\Arrow}[1]{%
\parbox{#1}{\tikz{\draw[->](0,0)--(#1,0);}}}
\usepackage[toc,page]{appendix}
\usepackage{array}
\newcolumntype{L}[1]{>{\raggedright\let\newline\\\arraybackslash\hspace{0pt}}m{#1}}
\newcolumntype{C}[1]{>{\centering\let\newline\\\arraybackslash\hspace{0pt}}m{#1}}
\newcolumntype{R}[1]{>{\raggedleft\let\newline\\\arraybackslash\hspace{0pt}}m{#1}}
\usepackage[english]{babel}
\usepackage{float}
\usepackage[super]{nth}
\usepackage{titling}
\usepackage{caption}
\usepackage{subcaption,dsfont}
\usepackage{multirow}
\usepackage{pbox}
\usepackage{fancyhdr, pdfpages, lipsum, accents}
\usepackage{authblk}
\newcommand{\sign}{\text{sign}}
\newtheorem{theorem}{Theorem}[section]

\newtheorem{proposition}[theorem]{Proposition}
\newtheorem{lemma}[theorem]{Lemma}

\newtheorem{remark}[theorem]{Remark}

\newtheorem{assumption}[theorem]{Assumption}
\newcommand{\E}{\mathbb{E}}
\newcommand{\R}{\mathbb{R}}

\newcommand{\Pro}{\mathbb{P}}
\newcommand{\X}{\mathbf{X}}
\newcommand{\x}{\mathbf{x}}

\newcommand{\W}{\mathbf{W}}

\numberwithin{equation}{section}
\title{Stochastic Maximum Principle for Optimal Liquidation with Control-dependent Terminal Time}
\author[1]{Riccardo Cesari\thanks{Department of Mathematics, Imperial College, London SW7 2BZ, UK. \href{riccardo.cesari17@imperial.ac.uk}{riccardo.cesari17@imperial.ac.uk} } \ and   Harry Zheng\thanks{Department of Mathematics, Imperial College, London SW7 2BZ, UK. \href{h.zheng@imperial.ac.uk}{h.zheng@imperial.ac.uk}. Supported in part by the EPSRC (UK) Grant (EP/V008331/1).   }}
\date{}

\begin{document}
\allowdisplaybreaks
\maketitle
\noindent\textbf{Abstract}. In this paper we study a general optimal liquidation problem with a control-dependent stopping time which is the first time the stock holding becomes zero or a fixed terminal time, whichever comes first. We prove a stochastic maximum principle (SMP) which is markedly different in its Hamiltonian condition from that of the standard SMP  with fixed terminal time. We present a simple example in which the optimal solution satisfies the SMP in this paper but fails the standard SMP in the literature. 



\section{Introduction}
Two main approaches to solving stochastic optimal control problems are the partial differential equation (PDE) method based on the dynamic programming principle (DPP) and the backward stochastic differential equation (BSDE) method based on the stochastic maximum principle (SMP), see Fleming and Soner~\cite{FlemingWendellH2006CMPa}, Yong and Zhou~\cite{YongZhou1999} for expositions. Most literature is on fixed terminal time problems. When the underlying state process is a controlled diffusion process, one may find the optimal solution with the Hamilton-Jacobi-Bellman (HJB) equation (a nonlinear PDE) or the BSDE (a coupled forward-backward system connected with the Hamiltonian condition). When the underlying state process is a controlled jump-diffusion process, one may find the optimal solution with the HJB partial differential difference equation or the BSDE, see Oksental and Sulem \cite{Oeksendal2007} for model formulations and solution methodologies. When the terminal time is not fixed but a random stopping time that is determined directly by a decision-maker,  one has an optimal stopping problem or a combined optimal control/stopping problem and may find the optimal solution with the HJB variational equation or the reflected BSDE, see Pham~\cite{PhamHuyen2009CSCa} for an excellent concise introduction of this and other topics above,  see also  Barbu and R\"{o}ckner~\cite{Barbu2016}, Cordoni et al.~\cite{Cordoni2020}, Diomande and Maticiuc~\cite{Diomande2017},  Popier and Zhou~\cite{Popier2019} for results on existence and uniqueness of forward backward and delayed stochastic differential equations and second order BSDEs and applications.

In this paper we investigate a stochastic optimal control problem with a random terminal time. In contrast to optimal stopping problem, the terminal time is indirectly determined by control strategies. Specifically, we consider an optimal liquidation problem in which the terminal time is determined by the first time the stock holding becomes zero or a fixed terminal time, whichever comes first. The objective is to maximize the expected cash value of the liquidation at the terminal time subject to some other underlying state process (stock price, volatility, etc.) dynamics. Such a model cannot be cast into the framework of the optimal stopping problem as stopping time is not directly controlled nor the jump-diffusion model with the fixed terminal time as stopping time is random. 
  For a Markovian model, one can show that the value function satisfies the HJB equation with the boundary condition when the stock holding is zero as well as the terminal condition at the fixed terminal time, 
see for example Cartea et al.~\cite{CarteaAlvaroauthor2015Aaht}. Since the HJB equation is a nonlinear PDE, one can find the closed-form solution or show the existence of a classical solution only for some specific models and has to rely on the viscosity solution concept for general models. For a non-Markovian model or with control constraints, the HJB approach loses its tractability.  On the other hand, the SMP approach provides an alternative way of solving the problem for general, possibly non-Markovian, models.  However, the standard SMP only applies to the problem with fixed terminal time. The first huddle we must overcome is to find the form of the BSDE and the associated Hamiltonian condition for the problem with random control-dependent terminal time.  

 There have been some efforts in the literature to address optimal trading. Ankirchner et al.~\cite{AnkirchnerStefan2014BwST} use the BSDE approach to solving the singular terminal state problem in which the terminal inventory is forced to be zero at terminal time, differently from our setting in which this constraint is weakened. Horst and Naujokat~\cite{Horst2014} state a version of the SMP  for optimal trading strategy with the spread driven by a jump diffusion process. Pham~\cite{PhamHuyen2010Scup} studies a model with multiple stopping times, a special form of the jump-diffusion process, and characterizes the value function with a system of backward recursive dynamic programming equations and the optimal control with progressive enlargement of filtration.

Cordoni and Di Persio~\cite{Cordoni2018} study a similar model to \cite{PhamHuyen2010Scup} and derive a system of backward recursive BSDEs that are similar to the standard SMP over adjacent stopping times intervals. There is, however, a key difference in the stopping time definitions in Pham~\cite{PhamHuyen2010Scup} and Cordoni and Di Persio~\cite{Cordoni2018}. The former is independent of controls and is given by some driving jump processes whereas the latter depends on controls and is given by the first time the underlying controlled state process hits some deterministic boundaries.   In the standard derivation of SMP with fixed terminal time (c.f. Bensoussan~\cite{Bensoussan2006} and Pham~\cite{PhamHuyen2009CSCa}), one may use the optimality condition and the variation of the optimal control to derive the BSDE and the Hamiltonian condition. When the terminal time is a stopping time depending on control, then variational analysis involves changes in terminal time as well as underlying state variables, unlike that for the standard SMP that only involves changes in underlying state variables. 
This aspect is the most arduous difficulty that needs to be overcome in the proof of the SMP with control dependent terminal time.
Cordoni and Di Persio~\cite{Cordoni2018} prove the SMP without discussing the possibility of changes of terminal stopping times due to changes of controls, see for example \cite[equation (15)]{Cordoni2018} in the proof of necessary SMP and \cite[equation (24)]{Cordoni2018} in the proof of sufficient SMP, which 
implies the SMP in \cite{Cordoni2018} is only valid for a model with stopping times {\it independent} of controls, same as that of \cite{PhamHuyen2010Scup}, but invalid for control-dependent stopping times, that is, the BSDE and the Hamiltonian condition in  \cite{Cordoni2018} are not applicable to the stopping time definition there.   

The main contribution of this paper is to give the SMP in the presence of random control-dependent terminal time, the first in the literature to the best knowledge of the authors. The main theorem (Theorem \ref{SMPtheorem}) states that the adjoint process satisfies a standard BSDE (see (\ref{defY})) with the terminal time the optimal stopping time, not necessarily the fixed terminal time, and the Hamiltonian function is also a standard one (see (\ref{defHcal})), but the Hamiltonian condition is markedly different from that of the standard  SMP with fixed terminal time (see (\ref{SMP})), specifically, we need to add a supplementary nonlinear term that is not additively separable between the optimal control and any other controls. This additional term counts for the random control-dependent terminal time. We give a simple example to show that the optimal solution satisfies the SMP in  Theorem \ref{SMPtheorem} but not the standard SMP, e.g., the one in \cite{PhamHuyen2009CSCa}. The SMP in this paper only applies to 
the optimal liquidation problem, but the idea of the variational analysis involving controlled stopping time may be explored further for more general models with other applications, for example, the optimal path planning problem that steers autonomous vehicles to navigate between any two points while optimizes energy, time, etc., see  Lee et al.~\cite{Lee2015} and Subramani et al~\cite{Subramani2018}.

The rest of the paper is organized as follows. In Section \ref{modelsetup} we describe the model and state the main result (Theorem \ref{SMPtheorem}) that is the SMP with random control-dependent terminal time. In Section \ref{sectionstochexample} we present an example to illustrate the main result and show the standard SMP does not hold.  In Section \ref{sectionproofs} we prove Theorem \ref{SMPtheorem}. Section \ref{sectionconclusion} concludes.

\section{Model setup}
\label{modelsetup}
Let $(\Omega, \mathcal{F}, (\mathcal{F}_t)_{t\in [0,T]}, \Pro)$ be a filtered probability space, where $(\mathcal{F}_t)_{t\in [0,T]}$ is the natural filtration generated by an $m$-dimensional standard Brownian motion $\W$, augmented by all $\Pro$-null sets. Let $T$ be the fixed terminal time. Let $(\pi_t)_{t\in[0,T]}$ denote the rate of selling the stock, which is a decision (control) variable selected by the agent and is said admissible if it is a progressively measurable, non-negative, right-continuous and square integrable process. Denote by $\mathcal{A}$ the set of all admissible control processes. We consider $\pi$ to be the liquidation rate of the inventory $Q_t$, defined as 
\begin{equation}
\label{defQ}
Q_t^\pi=q_0-\int^t_0 \pi_r \; dr .
\end{equation}
Let $(\X_t)_{t\in[0,T]}$ be an $\R^n$ valued stochastic process satisfying the following stochastic differential equation (SDE):
\begin{equation}
\label{defX}
d\X_t=\bm{\mu}(t,\X_t)dt+\bm{\sigma}(t,\X_t)d\W_t,
\end{equation}
with initial condition $\X_0=\x$, where $\bm{\mu}:[0,T]\times \R^n\to \R^n$, $\bm{\sigma}:[0,T]\times \R^n\to \R^{n\times m}$ are two continuous functions. $\X$ represents the market information such as stock price, volatility, etc., and is not influenced by the control process $\pi$  (liquidation without price impact). The optimal liquidation problem is defined by
\begin{equation}
\label{optproblem}
\sup_{\pi\in \mathcal{A}}\E\bigg[ g( \X_{\tau^\pi}, Q_{\tau^\pi}^\pi)+ \int_0^{\tau^\pi} f(r, \pi_r, \X_r, Q_r^\pi) \;dr \bigg],
\end{equation}
where $g:\R^n\times \R\to \R$ and $f:[0,T]\times [0,\infty)\times \R^n\times \R\to \R$ are two continuously differentiable functions representing the terminal and running payoffs, $\tau^\pi$ is a stopping time defined by
\begin{equation}
\label{deftau}
\tau^\pi=T\wedge \min\{r\ge 0\; | \; Q_r^\pi=0\},
\end{equation}
the first time when all stock is liquidated or fixed terminal time $T$, whichever comes first. 

To simplify the notation we consider a one-dimensional process $X$, but all results can be obtained in the multi-dimensional case. We denote the state space of the pair $(X_r,Q_r)$ as $\mathcal{O}:=\R\times[0,q_0]$. In the following, with a slight abuse of notations, we denote $\pi_r$ as $\pi_r\mathds{1}_{r\le\tau^\pi}$, which equals $0$ after $\tau^\pi$. Similarly, $Q^\pi_r=0$ for $r> \tau^\pi$. Whenever we refer to a time interval $[a, b)$, if $a\ge b$, then we consider it to be an empty set.

To state and prove a necessary SMP for problem $\eqref{optproblem}$,  we follow the procedure in Bensoussan~\cite{Bensoussan2006}. Assume that $c$ is the optimal control and $Q$ and $\tau$ are the corresponding optimal  inventory and stopping time,  defined in $\eqref{defQ}$ and $\eqref{deftau}$ respectively. If $Q_0=0$ then $\tau=0$ and there is nothing to discuss. We assume $Q_0>0$ which implies $\tau>0$. We next define the variation of the optimal control $c$.
 For fixed $t\in[0,\tau)$ and $\bar{c} \ge 0$, we have $q:=Q_t>0$. Choose $\theta\in\left(0,(T-t)\wedge\frac{q}{\bar{c}}\right)$,  that is, $t<t+\theta<T$ and $q-\bar c\theta>0$, consider a variation of $c$ as follows:
\begin{equation}
\label{defctheta}
c^{\theta,\bar{c}, t}_r:=c_r\mathds{1}_{r\in[0, t)} +   \bar{c}\mathds{1}_{r\in[t, (t+\theta)\wedge \tau)}+c_r\mathds{1}_{r\in[t+\theta,\tau)}-\frac{\gamma^{\theta,\bar{c}, t}_{t+\theta}}{\theta}\mathds{1}_{r\in[\tau, +\infty)},
\end{equation}
where  $\mathds{1}_S$ is an indicator that equals 1 if $S$ is true and 0 otherwise, and for $r\geq t$, 
\begin{equation}
\label{defgammatheta}
\gamma^{\theta,\bar{c}, t}_r:=\int_t^{r\wedge\tau}(\bar{c}-c_s)ds.
\end{equation}
The control $c^{\theta,\bar{c}, t}$ in $\eqref{defctheta}$ is  an admissible control, see Lemma \ref{lemmaoncthetaadmissible}. Let $Q^{\theta,\bar{c}, t}_r$ be the corresponding inventory under the control $c^{\theta,\bar{c}, t}$, given by  (\ref{defQ}), 
and  $\tau^{\theta,\bar{c}, t}$ be the first hitting time $r$ when the inventory $Q^{\theta,\bar{c}, t}_r$ gets to $0$,  given by (\ref{deftau}).

Since $c^{\theta,\bar{c}, t}_r=c_r$ for $r\in[0, t)$, we have $Q^{\theta,\bar{c}, t}_r=Q_r$ for $r\in[0, t)$, in particular, $Q^{\theta,\bar{c}, t}_t=Q_t$
and $\tau^{\theta,\bar{c}, t}>t$.  In fact, with the condition on $\theta$, $Q^{\theta,\bar{c}, t}_{(t+\theta)\wedge \tau} = Q_t - \bar c((t+\theta)\wedge \tau-t)\geq q-\bar c\theta>0$, so $\tau^{\theta,\bar{c}, t}>(t+\theta)\wedge \tau$.
The term $\gamma^{\theta,\bar{c}, t}_{t+\theta}$ represents the difference of the total liquidation on the time interval $[t,(t+\theta) \wedge\tau]$ with the constant control $\bar c$ and with the optimal control $c$, which determines the relation of $\tau$ and $\tau^{\theta,\bar{c}, t}$. If $\gamma^{\theta,\bar{c}, t}_{t+\theta}>0$ then $t+\theta<\tau^{\theta,\bar{c}, t}<\tau$ when $t+\theta<\tau<T$, specifically,  
$\tau^{\theta,\bar{c}, t}$ is  the time $r$ when the optimal inventory $Q_r$ is equal to $\gamma^{\theta,\bar{c}, t}_{t+\theta}$, see Figure \ref{figposeps}. 
If $\gamma^{\theta,\bar{c}, t}_{t+\theta}<0$ then $\tau^{\theta,\bar{c}, t}>\tau$, see Figure \ref{fignegeps}.
Note that we denote $c_r^{\theta,\bar{c}, t}$ as $c_r^{\theta,\bar{c}, t}\mathds{1}_{r\le\tau^{\theta,\bar{c}, t}}$, which equals $0$ after $\tau^{\theta,\bar{c}, t}$. We see that $c_r^{\theta,\bar{c}, t}$ is well defined in $\eqref{defctheta}$. Indeed, when $\tau\ge\tau^{\theta,\bar{c}, t}$, the last term in $\eqref{defctheta}$ disappears, while if $\tau<\tau^{\theta,\bar{c}, t}$, the quantity $\gamma^{\theta,\bar{c}, t}_{t+\theta}$ is negative, making the last term in $\eqref{defctheta}$ a non-negative term.

\begin{figure}[H]
\centering
\begin{subfigure}{.45\textwidth}
  \includegraphics[width=\textwidth]{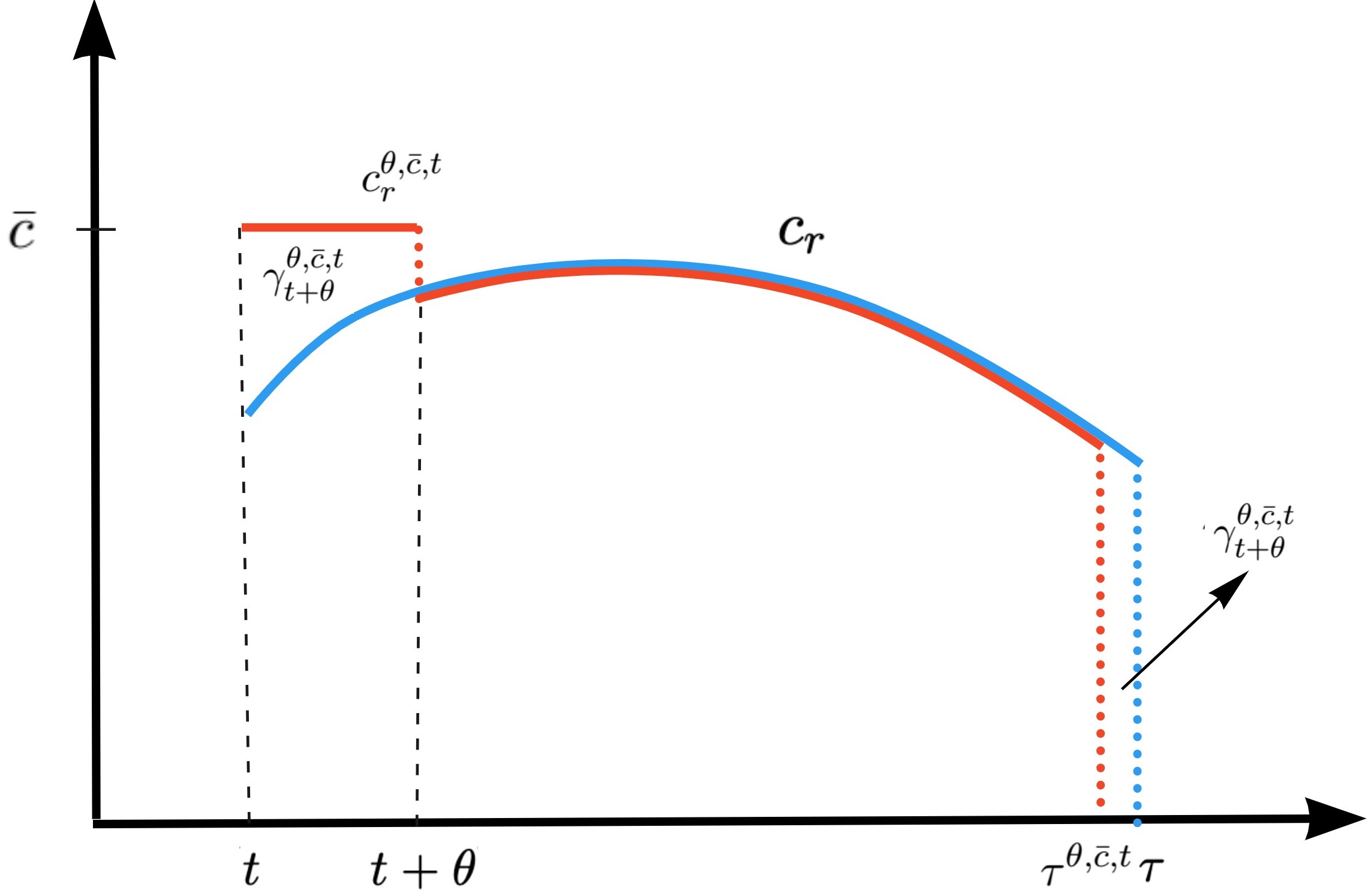}
  \caption{Case when  $\gamma^{\theta,\bar{c}, t}_{t+\theta}>0$.}
  \label{figposeps}
  \end{subfigure}%
\begin{subfigure}{.45\textwidth}
  \includegraphics[width=1.1\textwidth]{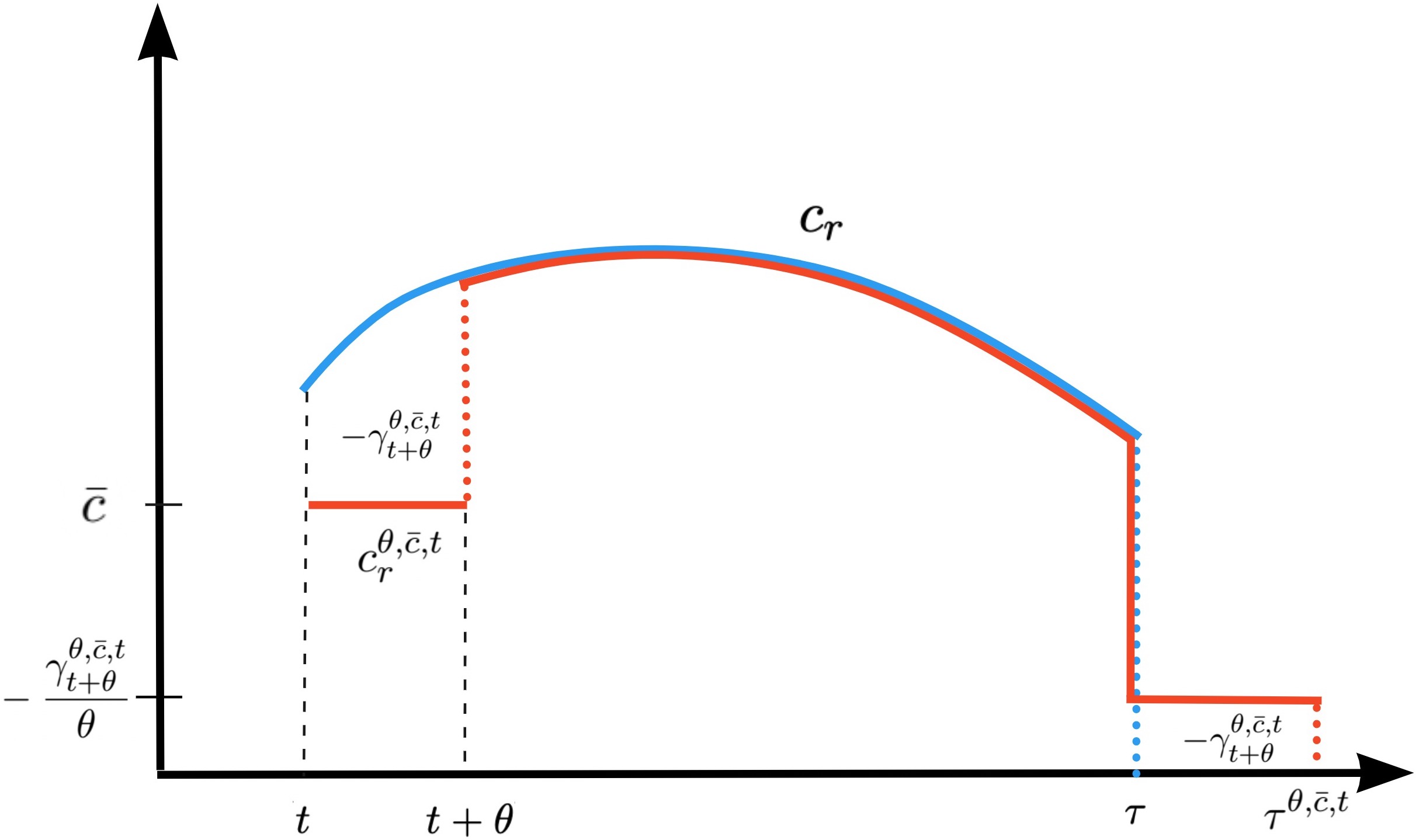}
  \caption{Case when  $\gamma^{\theta,\bar{c}, t}_{t+\theta}<0$.}
  \label{fignegeps}
  \end{subfigure}
\caption{Graphical examples of $c$ (in blue) and $c^{\theta,\bar{c}, t}$ (in red).}
\end{figure}
Let $(Y_r,Z_r)_{r\in[0,\tau]}$ be a solution of the following BSDE:
\begin{equation}
\label{defY}
\begin{cases}
-dY_r=\partial_q f(r, c_r,X_r,Q_r) dr-Z_r dW_r\\
Y_\tau=\partial_q g( X_\tau, Q_\tau).
\end{cases}
\end{equation}
The processes $(Y,Z)$ are conventional in the usual SMP formulation (c.f. Bensoussan~\cite{Bensoussan2006} and Pham~\cite{PhamHuyen2009CSCa}). BSDE $\eqref{defY}$ has a random terminal time $\tau$, in contrast to the standard BSDE with fixed terminal time.  This type of BSDEs has been studied in the literature, c.f. Darling and Pardoux~\cite{darling1997} and in Wu~\cite{wu_2003}.

 We assume $\mu, \sigma, f, g$ satisfy the following conditions for some  positive constant $K$.
\begin{assumption}
\label{secassumptions}
For any $t,t'\in[0,T]$, $\pi,\pi'\ge 0$, $x,x'\in\R$, $q,q'\ge 0$,
\begin{equation}
\label{assumptions}
\begin{split}
&\left| \mu(t,x)-\mu(t,x') \right| + \left| \sigma(t,x)-\sigma(t,x') \right|\le K\left| x-x' \right|,\\
&\left| \mu(t,x) \right| + \left| \sigma(t,x) \right|\le K\left( |x|+1 \right),\\
&\left| g(x,q)-g(x,q') \right| \le K(1+\left| x \right|)\left| q-q' \right|,\\
&\left| f(t,\pi,x,q)-f(t',\pi,x,q') \right| \le K\left(\left| q-q' \right|+\left| t-t' \right|\right),\\
&\left| f(t,\pi,x,q)-f(t,\pi',x',q) \right| \le K\left(\left| x-x' \right|+\left| \pi-\pi' \right|\right)\left(1+|x|+|x'| +|\pi|+|\pi'| \right),\\
&\left| \partial_q f(t,\pi,x,q)-\partial_q f(t,\pi,x,q') \right|+\left| \partial_q g(x,q)-\partial_q g(x,q') \right| \le K|q-q'|.
\end{split}
\end{equation}
\end{assumption}

Define the Hamiltonian as
\begin{equation}
\label{defHcal}
\mathcal{H}(t,\pi,x,q,y):=-\pi y+f(t,\pi,x,q).
\end{equation}
Denote by 
\begin{equation*}
\begin{split}
&\tau_{\min}^{\theta,\bar{c}, t}:=\min\left(\tau, \tau^{\theta,\bar{c}, t}\right),\ \tau_{\max}^{\theta,\bar{c}, t}:=\max\left(\tau, \tau^{\theta,\bar{c}, t}\right),\\
& \hat{Q}^{\theta,\bar{c}, t}_r:=\max\left(Q_r, Q^{\theta,\bar{c}, t}_r\right),\ \hat{c}^{\theta,\bar{c}, t}_r:=\max\left(c_r, c^{\theta,\bar{c}, t}_r\right).
\end{split}
\end{equation*}
We now state the stochastic maximum principle for problem $\eqref{optproblem}$.
\begin{theorem}
\label{SMPtheorem}
Let Assumption \ref{secassumptions} be satisfied. Let $(c_r)_{r\in[0,T]}$ be the optimal control for  problem $\eqref{optproblem}$, satisfying  
\begin{equation}
\label{asscquad}
\E\left[\sup_{r\in[0,T]}c_r^2\right]<\infty.
\end{equation}
Let $(Q_r)_{r\in[0,T]}$ and $(X_r)_{r\in[0,T]}$ be the corresponding solutions to SDEs $\eqref{defQ}$ and $\eqref{defX}$ with $Q_0>0$, $\tau>0$ the corresponding stopping time in (\ref{deftau}), and $(Y_r, Z_r)_{r\in[0,\tau]}$  the solution to BSDE $\eqref{defY}$. Assume that there exist $\R$-valued  functions $\bar{g}(t,\bar{c},x,q)$ and $\bar{f}(t,\bar{c},x,q)$ so that for any $t\in[0,\tau)$,  $(x,q)\in\mathcal{O}$ and  $\bar{c}\ge 0$,
\begin{align}
\label{defmubarpos}
\bar{g}(t,\bar{c},x,q)&=\lim_{\theta\to 0}\E^t\left[\frac{g(X_{\tau},Q^{\theta,\bar{c}, t}_{\tau^{\theta,\bar{c}, t}})-g(X_{\tau^{\theta,\bar{c}, t}},Q^{\theta,\bar{c}, t}_{\tau^{\theta,\bar{c}, t}})}{\theta}\right],\\
\label{deffbarpos}
\bar{f}(t,\bar{c},x,q)&=\lim_{\theta\to 0}\E^t\left[\frac{\sign(\tau-\tau^{\theta,\bar{c}, t})}{\theta}\int_{\tau_{\min}^{\theta,\bar{c}, t}}^{\tau_{\max}^{\theta,\bar{c}, t}}f\left(r, \hat{c}^{\theta,\bar{c}, t}_r, X_r,\hat{Q}^{\theta,\bar{c}, t}_r\right)\; dr\right],
\end{align}
where $\E^t[\cdot]=\E[\cdot | X_t=x, \ Q_t=q]$ is the conditional expectation  at time $t$. Then, $c$ necessarily satisfies for $t\in[0,\tau)$,  $\bar{c}\ge0$,
\begin{equation}
\label{SMP}
\mathcal{H}(t,\bar{c},X_t,Q_t,Y_t)-\mathcal{H}(t,c_t,X_t,Q_t,Y_t)+\mathcal{G}(t,\bar{c},X_t,Q_t)\le 0 \quad \text{a.s.},
\end{equation}
where $\mathcal{G}(t,\bar{c},x,q)$ is defined as
\begin{equation}
\label{defGcal}
\mathcal{G}(t,\bar{c},x,q):=(\bar{c}-c_t)\E^t\left[\partial_q g(X_\tau, Q_\tau) \mathds{1}_{\Lambda(t,\bar{c})}\right]-\bar{g}(t,\bar{c},x,q)-\bar{f}(t,\bar{c},x,q)
\end{equation}
and  $\Lambda(t,\bar{c})$ as
\begin{equation*}
\Lambda(t,\bar{c}):=\left(\{Q_T=0\}\cap \{\bar{c}\ge c_t\}\right) \cup \left(\{\tau<T\}\cap \{\bar{c}< c_t\}\right).
\end{equation*}
\end{theorem}

\begin{remark}
The definition of $\bar{g}$ in $\eqref{defmubarpos}$ is asymmetric in the arguments of functions $g$. We may define $\bar{g}$ in a symmetric way as
\begin{equation*}
\lim_{\theta\to 0}\E^t\left[\frac{g(X_{\tau},Q_\tau)-g(X_{\tau^{\theta,\bar{c}, t}},Q^{\theta,\bar{c}, t}_{\tau^{\theta,\bar{c}, t}})}{\theta}\right].
\end{equation*}
To get  an analogy of the Hamiltonian condition in Theorem \ref{SMPtheorem}, we have to define $\bar{g}$ as  in $\eqref{defmubarpos}$. This point is illustrated in $\eqref{derivationepspos}$ in the proof of Theorem \ref{SMPtheorem}.
 Note also that
stopping time $\tau$ is determined by the optimal control $c$ as in (\ref{deftau}) and  is therefore given.  The definitions of $\tau$ and $\tau^{\theta, \bar{c}, t}$ are unrelated but $\tau^{\theta, \bar{c}, t}$ converges to $\tau$ in $L^1$ and almost surely as $\theta \to 0$, see Lemma \ref{lemmaonconvergencetau}.
One interesting question raised by one of the reviewers is that if the limit $\lim_{\theta \to 0} \frac{\tau^{\theta, \bar{c},t}-\tau}{\theta}$ exists or not. The answer in general is negative. This can be seen by the following simple    example. Assume $t=0$ and $q_0=T^2/2$. Assume the optimal control is $c_t=T-t$ for $t\in [0,T]$, which gives 
$Q_r=q_0-\int_0^r c_s ds = (1/2)(T-r)^2$ for $r\in [0,T]$ and $Q_r=0$ if and only if $r=\tau:=T$.
Now consider a perturbation with $\bar{c}>T$ and $0<\theta<T^2/(2\bar c)$, which gives $Q^{\theta, \bar{c},t}_r=q_0-\bar c r>0$ for $r\in [0,\theta]$ and $Q^{\theta, \bar{c},t}_r=q_0-\int_0^r   c^{\theta, \bar{c},t}_s ds = 
-\bar{c}\theta +T\theta   -  \theta^2/2 + (T-r)^2/2$ for $r\in [\theta, T]$ and $Q^{\theta, \bar{c},t}_r=0$ if and only if
$r=\tau^{\theta, \bar{c},t}:=T-\sqrt{\theta^2-2T\theta+2\bar{c}\theta}$. We have $\tau^{\theta, \bar{c},t}\to \tau$ as $\theta\to 0$ but
\begin{equation*}
\lim_{\theta \to 0} \frac{\tau^{\theta, \bar{c},t}-\tau}{\theta}=-\lim_{\theta \to 0} \frac{\sqrt{\theta^2-2T\theta+2\bar{c}\theta}}{\theta}=-\lim_{\theta \to 0} \sqrt{1+2\frac{\bar{c}-T}{\theta}}=-\infty,
\end{equation*}
which shows the limit does not exist. 
\end{remark}

\begin{remark}
The standard SMP (c.f. Pham~\cite{PhamHuyen2009CSCa})  and Theorem \ref{SMPtheorem} cannot be recovered from each other as they solve different problems. However, their statements are similar and only differ for the additional term $\mathcal{G}$ in \eqref{defGcal}.
\end{remark}

\begin{remark}
The same result as Theorem \ref{SMPtheorem} can be obtained in the case when the admissible set is bounded by above as well, i.e. when $\pi$ is required to be in $\pi\in[0,b]$ with $0<b\le +\infty$. Although the proof does not change, the only remark we want to point out is on the admissibility of control $c^{\theta,\bar{c}, t}$. Since $\bar{c}\in[0,b]$ and $c_r\in[0,b]$ for every $r\in[0,T]$, then $-\frac{\gamma^{\theta,\bar{c}, t}_{t+\theta}}{\theta}=\frac{1}{\theta}\int_t^{(t+\theta)\wedge \tau}(c_r-\bar{c})dr\le \frac{(t+\theta)\wedge \tau-t}{\theta}b\le b$.
\end{remark}

\section{Example}
\label{sectionstochexample}
In this section we describe an example to show that the usual  SMP is not satisfied whereas Theorem \ref{SMPtheorem} is satisfied.  We consider an optimal liquidation problem with no market impact on trade and no terminal execution. In particular, let $g=0$ and $f(\pi,x)=\pi x$. Assume the admissible control can only take values in the interval $[0,c^+]$. Let $t\in[0,T]$ and $Q_t=q>0$ be fixed. The stock price $X$ satisfies the SDE:  for $r\in[t,T]$ and $X_t=x>0$,
\begin{equation*}
dX_r=X_rdW_r.
\end{equation*}
 The agent aims to maximise his final cash value that is the cumulated liquidation wealth up to the stopping time at which the agent runs out the liquidating stocks and there is no residual value for any remaining stocks at horizon time $T$. The value function to this problem is defined by
\begin{align}
\label{valfuncexamplestoch}
v(t,x,q)&=\sup_{\pi\in \mathcal{A}}\E^t\left[ \int_t^{\tau^\pi} \pi_r X_r \;dr \right],
\end{align}
where  $\tau^\pi=T\wedge \min\{r\ge t\; | \; Q^\pi_r=0\}$ is the first hitting time of $Q^\pi$ to zero or the fixed terminal time $T$, whichever comes first. We define an admissible control strategy $c$ as follows: for any $r\in[t,T]$,
\begin{equation}
\label{optcexamplestoch}
c_r= \begin{cases}\frac{q}{T-t} & \text{if } q\le c^+(T-t),\\
c^+  & \text{if } q>c^+(T-t).
\end{cases}
\end{equation}
The inventory $Q_r$ in $\eqref{defQ}$ is given by 
\begin{equation*}
Q_r=\begin{cases}\frac{q}{T-t}(T-r) & \text{if } q\le c^+(T-t),\\
q-(r-t)c^+  & \text{if } q>c^+(T-t).
\end{cases}
\end{equation*}
We also have that for any $r\in[t,T]$,  $Q_r\le c^+(T-r) \;\Leftrightarrow\; q\le c^+(T-t)$. Using the above expression for $Q_r$, it is easy to check that the first hitting time of $Q_r=0$ is
\begin{equation*}
\tau=T \quad \text{ a.s.}.
\end{equation*}
Since the stopping time $\tau$ is equal to the terminal time $T$,  it may look like the control in $\eqref{optcexamplestoch}$ is the same control we would have found in the usual setting without the stopping time. However, the optimal control in the usual case without stopping time  would have been equal to $c^+$. When $q< c^+(T-t)$, at time $r=t+\frac{q}{c^+}< T$, the inventory would have reached zero  and from that time onward the inventory would have become negative, making the control $c^+$ not feasible for our problem, see Remark \ref{remarkexamplestoch}.

Substituting  $\eqref{optcexamplestoch}$ and $\tau=T$ into $\eqref{valfuncexamplestoch}$, we can easily show 
that  the value function associated with control $\eqref{optcexamplestoch}$ is equal to
\begin{equation}
\label{valfuncexstoch}
v^c(t,x,q)=\begin{cases}qx & \text{if } q\le c^+(T-t),\\
xc^+ (T-t) & \text{if } q>c^+(T-t).
\end{cases}
\end{equation}
\begin{proposition}
Function $v^c$ in $\eqref{valfuncexstoch}$ coincides with the value function $v$ of problem $\eqref{valfuncexamplestoch}$   and the strategy $c$ in (\ref{optcexamplestoch}) is the optimal control.
\end{proposition}
\begin{proof}
By definition of $v^c(t,x,q)$, we have that for any $t\in[0,T]$, $x>0$ and $q>0$
\begin{equation}
\label{eqqless1exstoch}
v^c(t,x,q)=\E^t\left[\int_t^\tau c_rX_r dr\right] \le \sup_{\pi\in\mathcal{A}}\E^t\left[\int_t^{\tau^\pi} \pi_rX_r dr\right] = v(t,x,q).
\end{equation}
We now show that $v^c\ge v$. Define
\begin{equation}
\label{secondvalfuncexstoch}
w(t,q)=\begin{cases}q & \text{if } q\le c^+(T-t),\\
c^+ (T-t) & \text{if } q>c^+(T-t).
\end{cases}
\end{equation}
 Simple calculus shows that
\begin{align*}
\partial_q w(t,q)=\begin{cases}1 & \text{if } q\le c^+(T-t),\\
0 & \text{if } q>c^+(T-t),
\end{cases}\qquad \partial_t w(t,q)=\begin{cases}0 & \text{if } q\le c^+(T-t),\\
-c^+ & \text{if } q>c^+(T-t).
\end{cases}
\end{align*}
It can be easily verified that $w$ satisfies the following HJB equation, for $t\in[0,T]$ and $q>0$
\begin{equation}
\label{HJBeqexamplestoch}
\partial_t w+\sup_{\pi\in[0,c^+]} \left[\pi-\pi\partial_qw\right]=0.
\end{equation}
$w$ satisfies the boundary condition $w(t,0)=0$ for any $t\in[0,T]$ and the terminal condition $w(T,q)=0$ for any $q>0$.

Denote by $c^*_r$ the optimal control for the value function $v(t,x,q)$ in $\eqref{valfuncexamplestoch}$ and $Q^*_r$ and $\tau^*$  the corresponding inventory and stopping time respectively. Using $\eqref{defQ}$, for any $r\in[t,T]$, we have
\begin{equation}
\label{eqqpropex2stoch}
dw(r,Q^*_r)=\partial_t w(r,Q^*_r)dr+\partial_qw(r,Q^*_r)dQ^*_r=\left[\partial_t w(r,Q^*_r)-\partial_qw(r,Q^*_r)c^*_r\right]dr\le -c^*_rdr.
\end{equation}
We have used (\ref{HJBeqexamplestoch}) in the last inequality. 
Finally, from  $X_t=x$, $Q_t^*=q$, $v^c(t,x,q)=xw(t,q)$, using the stochastic integration by parts and $\eqref{eqqpropex2stoch}$, we have
\begin{equation}
\label{eqq7571stoch}
\begin{split}
v^c(\tau^*,X_{\tau^*},Q^*_{\tau^*})&=v^c(t,X_t,Q^*_t)+\int_t^{\tau^*} X_r dw(r,Q^*_r) +\int_t^{\tau^*} w(r,Q^*_r) dX_r \\
&\le v^c(t,x,q)-\int_t^{\tau^*} X_r c^*_r dr+\int_t^{\tau^*} w(r,Q^*_r) X_r dW_r.
\end{split}
\end{equation}
However, using boundary and terminal condition $v^c(t,x,0)=0$ for any $t\in[0,T]$ and $x>0$ and $v^c(T,x,q)=0$ for any $x>0$ and $q>0$, we conclude that $v^c(\tau^*,X_{\tau^*},Q^*_{\tau^*})=0$, as either $\tau^*=T$ or $Q^*_{\tau^*}=0$. Therefore, taking conditional expectations on both sides of $\eqref{eqq7571stoch}$ and using the optional sampling theorem, we have that the expected value of the random variable $\int_t^{\tau^*} w(r,Q^*_r) X_r dW_r$ is equal to $0$ and we get
\begin{equation*}
v^c(t,x,q)\ge \E^t\left[\int_t^{\tau^*} c^*_rX_r\; dr\right] = v(t,x,q).
\end{equation*}
Combining $\eqref{eqqless1exstoch}$ and the previous expression, we conclude the proof.
\end{proof}

\begin{remark}
\label{remarkexamplestoch}
In the standard version of the SMP (cf. Pham~\cite{PhamHuyen2009CSCa}) there should be 2 adjoint processes in the BSDE, referring respectively to processes $X$ and $Q$. However, since the process $X$ does not depend on control $\pi$, the terms regarding the adjoint process referring to $X$ can be removed from the Hamiltonian and, noting that $g=0$ and $f(\pi,x)=\pi x$, those referring to $Q$ are identically equal to $0$. The necessary condition of the standard SMP is equivalent to
\begin{equation}
\label{eqqexstochSMP}
 f(\pi,X_t)\le f(c_t,X_t), \quad\forall t\in[0,T], \ \forall \pi\in[0,c^+].
\end{equation}
However, from $\eqref{valfuncexamplestoch}$ we get that the maximal point of $f(\cdot,x)$ is $\bar{\pi}=c^+$.  We have shown that if $q< c^+(T-t)$, then the optimal strategy is $c_r=\frac{q}{T-t}$ for $t\le r\le T$ as in $\eqref{optcexamplestoch}$, which is less than $\bar{\pi}=c^+$, a contradiction to $\eqref{eqqexstochSMP}$ and  the standard SMP. 
\end{remark}

We next verify  that optimal control in $\eqref{optcexamplestoch}$ satisfies Theorem \ref{SMPtheorem}. We need to show that $\eqref{SMP}$ holds true. Firstly, we observe that the model setup satisfies Assumption \ref{secassumptions}. Using the fact that $\mu= 0$ and  $g=0$, we show that $\eqref{SMP}$ holds true by proving that for any $\bar{c}\in[0,c^+]$, $t\in[0,T)$,
\begin{equation}
\label{toproveexstoch}
f(\bar{c},X_t)-f(c_{t},X_t)-\bar{f}(t,\bar{c},X_t,Q_t)\le 0.
\end{equation}

We first find the expression for $\bar{f}(t,\bar{c},x,q)$ in the following proposition. 
\begin{proposition}
Let $t\in[0,T)$ be fixed and  $(c_r)_{r\in[t,T]}$ be the optimal control in $\eqref{optcexamplestoch}$. Then, for any $\bar{c}\in[0,c^+]$, $x>0$ and $q>0$
\begin{equation*}
\bar{f}(t,\bar{c},x,q)=\begin{cases}(\bar{c}-c_t) x \mathds{1}_{\bar{c}\ge c_t} & \text{if } q\le c^+(T-t),\\
0 & \text{if } q>c^+(T-t).
\end{cases}
\end{equation*}
\end{proposition}
\begin{proof}
We consider any $\theta\in(0,T-t)$, so that $\tau=T> t+\theta$.  Using the fact that $c_r$ in $\eqref{optcexamplestoch}$ is constant in time, we have for any $r\in[t+\theta, \tau]$
\begin{equation}
\label{eq1stochex}
Q^{\theta,\bar{c}, t}_r=q-\bar{c}\theta-\int_{t+\theta}^r c_s ds=q-c_t(r-t)+\theta(c_t-\bar{c}).
\end{equation}
If $q>c^+(T-t)$, then $\eqref{optcexamplestoch}$ implies that $c_t=c^+$ and from $\eqref{eq1stochex}$, 
\begin{equation*}
Q^{\theta,\bar{c}, t}_T> \theta(c^+-\bar{c})\ge 0.
\end{equation*}
Here we have used the fact that $\bar{c}$ is an admissible control and so $\bar{c}\in[0,c^+]$. The above expression implies that $\tau^{\theta,\bar{c}, t}=T$ a.s.. On the other hand, if $q\le c^+(T-t)$, then $c_t=\frac{q}{T-t}$ and from $\eqref{eq1stochex}$, $Q^{\theta,\bar{c}, t}_T=-\theta(\bar{c}-c_t)$. Hence, if $\bar{c}\le c_t, Q^{\theta,\bar{c}, t}_T\ge 0$ and so $\tau^{\theta,\bar{c}, t}=T$ a.s.. If $\bar{c}> c_t, Q^{\theta,\bar{c}, t}_T< 0$ and so $\tau^{\theta,\bar{c}, t}<T$ a.s., so by setting $\eqref{eq1stochex}$ equal to $0$ we get that $\tau^{\theta,\bar{c}, t}=T-\theta\left(\frac{\bar{c}}{c_t}-1\right)$ a.s., where $\theta\left(\frac{\bar{c}}{c_t}-1\right)>0$, since $\bar{c}>c_t$.

In conclusion, if $q>c^+(T-t)$, then $\tau^{\theta,\bar{c}, t}_{\min}=\tau^{\theta,\bar{c}, t}_{\max}=T$ and  we have that $\bar{f}=0$ from definition $\eqref{deffbarpos}$. If $q\le c^+(T-t)$, then we consider two sub-cases. If $\bar{c}\le c_t$, then $\tau^{\theta,\bar{c}, t}_{\min}=\tau^{\theta,\bar{c}, t}_{\max}=T$, making $\bar{f}=0$ again. If $\bar{c}>c_t$, then $\tau^{\theta,\bar{c}, t}_{\min}=T-\theta\left(\frac{\bar{c}}{c_t}-1\right)$ and $\tau^{\theta,\bar{c}, t}_{\max}=T$ and so
\begin{align*}
\bar{f}(t,\bar{c},x,q)&=\lim_{\theta\to 0}\E^t\left[\frac{1}{\theta}\int_{T-\theta\left(\frac{\bar{c}}{c_t}-1\right)}^{T}c_r X_r\; dr\right]=\lim_{\theta\to 0}\frac{\theta\left(\frac{\bar{c}}{c_t}-1\right)}{\theta}c_tx=(\bar{c}-c_t) x .
\end{align*}
This concludes the proof of the proposition.
\end{proof}
To prove $\eqref{toproveexstoch}$,  we split the proof of $\eqref{toproveexstoch}$ in two parts.  If $q\le c^+(T-t)$, then the left side of $\eqref{toproveexstoch}$ is equal to $\bar{c} X_t-c_t X_t-(\bar{c}-c_t)X_t\mathds{1}_{\bar{c}\ge c_t}=X_t\cdot\min\left(\bar{c}-c_t, \ 0\right)\leq 0$. If $q> c^+(T-t)$, then the left  side of $\eqref{toproveexstoch}$ is equal to $\bar{c} X_t-c_t X_t=(\bar{c}-c^+)X_t\leq 0$ as  $\bar{c}\le c^+$. Hence, $\eqref{toproveexstoch}$ is satisfied for any $c^+\ge \bar{c}\ge 0$, and Theorem \ref{SMPtheorem} holds.


\begin{remark}
The main purpose of this example is to show that the standard SMP cannot be applied when the terminal time is an indirectly controlled stopping time but our necessary SMP can accommodate that. It is in general difficult to find the optimal solution using Theorem \ref{SMPtheorem} that has no particular advantage to the DPP  for the example, but this is the first step in  addressing  the indirectly controlled random terminal time problem with the SMP which has the potential for solving  the non-Markovian model,  see Section 5 for possible further research. 
\end{remark}

\section{Proof of Theorem \ref{SMPtheorem}}
\label{sectionproofs}
In this section we consider all assumptions of Theorem \ref{SMPtheorem} are satisfied. $\eqref{asscquad}$ implies that $\E\left[\sup_{r\in[0,T]}c_r\right]<\infty$. As  mentioned in the model setup, for any fixed  time $t\in[0,\tau)$, we have  $Q_t=q>0$ and so $\tau>t$ a.s.. We consider a partition of the whole event space $\{\tau>t\}$, which helps us in stating and proving some preliminary results that are needed in the proof of Theorem \ref{SMPtheorem}. As general hints for better understanding, we remind that $\tau$ is defined so that $Q_\tau=0$ if $\tau<T$, $Q_\tau\ge 0$ if $\tau=T$, and $Q_r>0$ if $r\in[t,\tau)$. Similarly, $Q^{\theta,\bar{c}, t}_{\tau^{\theta,\bar{c}, t}}=0$  if $\tau^{\theta,\bar{c}, t}<T$,  $Q^{\theta,\bar{c}, t}_{\tau^{\theta,\bar{c}, t}}\ge 0$ if $\tau^{\theta,\bar{c}, t}=T$, and $Q^{\theta,\bar{c}, t}_r>0$ if $r\in[t,\tau^{\theta,\bar{c}, t})$. We first observe, using $\eqref{defctheta}$, that if $\theta\in\left(0, (T-t)\wedge\frac{q}{\bar{c}}\right)$, then for any $r\in[t,(t+\theta)\wedge \tau]$,
\begin{equation*}
Q^{\theta,\bar{c}, t}_r=q-\int_{t}^r c^{\theta,\bar{c}, t}_s ds=q-\bar{c}\left(r-t\right)\ge q-\bar{c}\theta>0.
\end{equation*}
Therefore, if $\theta\in\left(0, (T-t)\wedge\frac{q}{\bar{c}}\right)$, then 
\begin{equation}
\label{eqtauthetabig}
\tau^{\theta,\bar{c}, t} >(t+\theta)\wedge \tau
\end{equation}
Let $t\in[0,\tau)$, $0< \theta<(T-t)\wedge\frac{q}{\bar{c}} $, $\bar{c}\ge 0$ be fixed, we define the following partitions of $\{\tau>t\}$:

\begin{equation*}
\begin{split}
E^{\theta,\bar{c}, t}_{1}&:=\left\{ t<t+\theta<\tau^{\theta,\bar{c}, t}<\tau \right\},\\
E^{\theta,\bar{c}, t}_{2}&:=\left\{ t<\tau<\tau^{\theta,\bar{c}, t} \right\},\\
E^{\theta,\bar{c}, t}_{3}&:=\left\{ t< t+\theta<\tau=\tau^{\theta,\bar{c}, t} \right\}.
\end{split}
\end{equation*}
We now present the properties of the different cases $E^{\theta,\bar{c}, t}_{i}$, for any $i\in\{1,2,3\}$. In particular, for each of the events we show a scheme for the different values of quantities $c^{\theta,\bar{c}, t}$ and $Q^{\theta,\bar{c}, t}$ in each of the time spans. These schemes help in understanding some steps in the proof of lemmas below.
\begin{enumerate}[label=\textbf{\arabic*})]

\item \textbf{On the event} $E^{\theta,\bar{c}, t}_{1}$:
\begin{table}[H]
\begin{center}
\begin{tabular}{c|c|c|c|c|c|c|c|c}
\multicolumn{1}{C{2cm}}{}&\multicolumn{2}{C{3cm}}{\kern-1em $t$}&\multicolumn{2}{C{3cm}}{\kern-1em$t+\theta$}&\multicolumn{2}{C{3cm}}{$\tau^{\theta,\bar{c}, t}$}&\multicolumn{2}{C{3cm}}{\kern-0.5em $\tau$}\\[-0.5ex]
\multicolumn{2}{R{3cm}|}{$ r\in $}&\multicolumn{2}{C{3cm}|}{}&\multicolumn{2}{C{3cm}|}{ }&\multicolumn{2}{C{3cm}|}{ }&\kern-5em \Arrow{1.5cm}\\ [-1.5ex]\cline{3-8}
\multicolumn{9}{C{15cm}}{ }\\
\multicolumn{2}{C{3.5cm}|}{$c_r^{\theta,\bar{c}, t}=$}&\multicolumn{2}{C{3cm}|}{$\bar{c}$}&\multicolumn{2}{C{3cm}|}{$c_r$}&\multicolumn{2}{C{3cm}|}{$0$}&\multicolumn{1}{C{1.5cm}}{}\\  
\multicolumn{2}{C{3.5cm}|}{$Q_r^{\theta,\bar{c}, t}=$}&\multicolumn{2}{C{3cm}|}{$Q_r-\gamma^{\theta,\bar{c}, t}_r$}&\multicolumn{2}{C{3cm}|}{$Q_r-\gamma^{\theta,\bar{c}, t}_{t+\theta}$}&\multicolumn{2}{C{3cm}|}{$0$}&\multicolumn{1}{C{1.5cm}}{}\\  
\multicolumn{2}{C{3.5cm}|}{$c_r^{\theta,\bar{c}, t}-c_r=$}&\multicolumn{2}{C{3cm}|}{$\bar{c}-c_r$}&\multicolumn{2}{C{3cm}|}{$0$}&\multicolumn{2}{C{3cm}|}{$-c_r$}&\multicolumn{1}{C{1.5cm}}{}\\
\multicolumn{2}{C{3.5cm}|}{$Q_r^{\theta,\bar{c}, t}-Q_r=$}&\multicolumn{2}{C{3cm}|}{$-\gamma^{\theta,\bar{c}, t}_r$}&\multicolumn{2}{C{3cm}|}{$-\gamma^{\theta,\bar{c}, t}_{t+\theta}$}&\multicolumn{2}{C{3cm}|}{$-Q_r$}&\multicolumn{1}{C{1.5cm}}{}\\  
\end{tabular}
\end{center}
\end{table}
From previous scheme we conclude that on the event $E^{\theta,\bar{c}, t}_{1}$
\begin{equation}
\label{QtauE1}
0=Q_{\tau^{\theta,\bar{c}, t}}^{\theta,\bar{c}, t}=Q_{\tau^{\theta,\bar{c}, t}}-\gamma^{\theta,\bar{c}, t}_{t+\theta} \ \Rightarrow \ Q_{\tau^{\theta,\bar{c}, t}}=\gamma^{\theta,\bar{c}, t}_{t+\theta}, \text{ which also implies that } \gamma^{\theta,\bar{c}, t}_{t+\theta}>0,
\end{equation}
since by definition of $\tau$, for any  $r\in[t,\tau)$, $Q_r>0$.
\begin{equation}
\label{absQminQE1}
\left|Q_r^{\theta,\bar{c}, t}-Q_r\right|\le \max\left(\sup_{r\in[t,t+\theta]}\left|\gamma^{\theta,\bar{c}, t}_r\right|,\ |Q_{\tau^{\theta,\bar{c}, t}}|\right)=\sup_{r\in[t,t+\theta]}\left|\gamma^{\theta,\bar{c}, t}_r\right|, \qquad \forall r\in[t,T],
\end{equation}
\begin{equation}
\label{QminQE1}
Q_{\tau^{\theta,\bar{c}, t}}-Q_{\tau}\le Q_{\tau^{\theta,\bar{c}, t}}= \gamma^{\theta,\bar{c}, t}_{t+\theta}.
\end{equation}

\item \textbf{On the event} $E^{\theta,\bar{c}, t}_{2}$: If $\tau> t+\theta$
\begin{table}[H]
\begin{center}
\begin{tabular}{c|c|c|c|c|c|c|c|c|c|c}
\multicolumn{1}{C{2cm}}{}&\multicolumn{2}{C{2cm}}{\kern-1em $t$}&\multicolumn{2}{C{2cm}}{\kern-1em$t+\theta$}&\multicolumn{2}{C{2cm}}{\kern-1em$\tau$}&\multicolumn{2}{C{3cm}}{\kern+1em $\tau^{\theta,\bar{c}, t}$}&\multicolumn{2}{C{2cm}}{\kern-0.5em $T$}\\[-0.5ex]
\multicolumn{2}{R{2cm}|}{$ r\in $}&\multicolumn{2}{C{2cm}|}{}&\multicolumn{2}{C{2cm}|}{ }&\multicolumn{2}{C{3cm}|}{ }&\multicolumn{2}{C{2cm}|}{}&\kern-2em \Arrow{1.5cm}\\ [-1.5ex]\cline{3-10}
\multicolumn{11}{C{16cm}}{ }\\
\multicolumn{2}{C{3cm}|}{$c_r^{\theta,\bar{c}, t}=$}&\multicolumn{2}{C{2cm}|}{$\bar{c}$}&\multicolumn{2}{C{2cm}|}{$c_r$}&\multicolumn{2}{C{3cm}|}{$-\frac{\gamma^{\theta,\bar{c}, t}_{t+\theta}}{\theta}$}&\multicolumn{2}{C{2cm}|}{$0$}&\multicolumn{1}{C{1cm}}{}\\  
\multicolumn{2}{C{3cm}|}{$Q_r^{\theta,\bar{c}, t}=$}&\multicolumn{2}{C{2cm}|}{$Q_r-\gamma^{\theta,\bar{c}, t}_{r}$}&\multicolumn{2}{C{2cm}|}{$Q_r-\gamma^{\theta,\bar{c}, t}_{t+\theta}$}&\multicolumn{2}{C{3cm}|}{$-\gamma^{\theta,\bar{c}, t}_{t+\theta}\left(1-\frac{r-\tau}{\theta}\right)$}&\multicolumn{2}{C{2cm}|}{$0$}&\multicolumn{1}{C{1cm}}{}\\  
\multicolumn{2}{C{3cm}|}{$c_r^{\theta,\bar{c}, t}-c_r=$}&\multicolumn{2}{C{2cm}|}{$\bar{c}-c_r$}&\multicolumn{2}{C{2cm}|}{$0$}&\multicolumn{2}{C{3cm}|}{$-\frac{\gamma^{\theta,\bar{c}, t}_{t+\theta}}{\theta}$}&\multicolumn{2}{C{2cm}|}{$0$}&\multicolumn{1}{C{1cm}}{}\\  
\multicolumn{2}{C{3cm}|}{$Q_r^{\theta,\bar{c}, t}-Q_r=$}&\multicolumn{2}{C{2cm}|}{$-\gamma^{\theta,\bar{c}, t}_{r}$}&\multicolumn{2}{C{2cm}|}{$-\gamma^{\theta,\bar{c}, t}_{t+\theta}$}&\multicolumn{2}{C{3cm}|}{$-\gamma^{\theta,\bar{c}, t}_{t+\theta}\left(1-\frac{r-\tau}{\theta}\right)$}&\multicolumn{2}{C{2cm}|}{$0$}&\multicolumn{1}{C{1cm}}{}\\  
\end{tabular}
\end{center}
\end{table}
If $\tau\le t+\theta$, from $\eqref{defgammatheta}$, $\gamma^{\theta,\bar{c}, t}_\tau=\gamma^{\theta,\bar{c}, t}_{t+\theta}$
\begin{table}[H]
\begin{center}
\begin{tabular}{c|c|c|c|c|c|c|c|c}
\multicolumn{1}{C{2cm}}{}&\multicolumn{2}{C{3cm}}{\kern-1em $t$}&\multicolumn{2}{C{3cm}}{\kern-1em$\tau$}&\multicolumn{2}{C{3cm}}{$\tau^{\theta,\bar{c}, t}$}&\multicolumn{2}{C{3cm}}{\kern-0.5em $T$}\\[-0.5ex]
\multicolumn{2}{R{3cm}|}{$ r\in $}&\multicolumn{2}{C{3cm}|}{}&\multicolumn{2}{C{3cm}|}{ }&\multicolumn{2}{C{3cm}|}{ }&\kern-5em \Arrow{1.5cm}\\ [-1.5ex]\cline{3-8}
\multicolumn{9}{C{15cm}}{ }\\
\multicolumn{2}{C{3.5cm}|}{$c_r^{\theta,\bar{c}, t}=$}&\multicolumn{2}{C{3cm}|}{$\bar{c}$}&\multicolumn{2}{C{3cm}|}{$-\frac{\gamma^{\theta,\bar{c}, t}_{t+\theta}}{\theta}$}&\multicolumn{2}{C{3cm}|}{$0$}&\multicolumn{1}{C{1.5cm}}{}\\  
\multicolumn{2}{C{3.5cm}|}{$Q_r^{\theta,\bar{c}, t}=$}&\multicolumn{2}{C{3cm}|}{$Q_r-\gamma^{\theta,\bar{c}, t}_r$}&\multicolumn{2}{C{3cm}|}{$-\gamma^{\theta,\bar{c}, t}_{t+\theta}\left(1-\frac{r-\tau}{\theta}\right)$}&\multicolumn{2}{C{3cm}|}{$0$}&\multicolumn{1}{C{1.5cm}}{}\\  
\multicolumn{2}{C{3.5cm}|}{$c_r^{\theta,\bar{c}, t}-c_r=$}&\multicolumn{2}{C{3cm}|}{$\bar{c}-c_r$}&\multicolumn{2}{C{3cm}|}{$-\frac{\gamma^{\theta,\bar{c}, t}_{t+\theta}}{\theta}$}&\multicolumn{2}{C{3cm}|}{$0$}&\multicolumn{1}{C{1.5cm}}{}\\
\multicolumn{2}{C{3.5cm}|}{$Q_r^{\theta,\bar{c}, t}-Q_r=$}&\multicolumn{2}{C{3cm}|}{$-\gamma^{\theta,\bar{c}, t}_r$}&\multicolumn{2}{C{3cm}|}{$-\gamma^{\theta,\bar{c}, t}_{t+\theta}\left(1-\frac{r-\tau}{\theta}\right)$}&\multicolumn{2}{C{3cm}|}{$0$}&\multicolumn{1}{C{1.5cm}}{}\\  
\end{tabular}
\end{center}
\end{table}
From previous scheme we conclude that on the event $E^{\theta,\bar{c}, t}_{2}$,
\begin{equation}
\label{gammanegE2}
Q^{\theta,\bar{c}, t}_\tau=Q_\tau-\gamma^{\theta,\bar{c}, t}_{t+\theta}=-\gamma^{\theta,\bar{c}, t}_{t+\theta}, \text{ which implies } \gamma^{\theta,\bar{c}, t}_{t+\theta}<0,
\end{equation}
since by definition of $\tau^{\theta,\bar{c}, t}$, for any $r\in[t,\tau^{\theta,\bar{c}, t})$, $Q^{\theta,\bar{c}, t}_r>0$. Moreover,
\begin{equation}
\label{QtauE2}
\begin{split}
\text{if } \tau^{\theta,\bar{c}, t}<T, \ 0=Q_{\tau^{\theta,\bar{c}, t}}^{\theta,\bar{c}, t}=-\gamma^{\theta,\bar{c}, t}_{t+\theta}\left(1-\frac{\tau^{\theta,\bar{c}, t}-\tau}{\theta}\right) \ \Rightarrow \ \tau^{\theta,\bar{c}, t}=\tau+\theta,\\
\text{if } \tau^{\theta,\bar{c}, t}=T, \ 0\le Q_{\tau^{\theta,\bar{c}, t}}^{\theta,\bar{c}, t}=-\gamma^{\theta,\bar{c}, t}_{t+\theta}\left(1-\frac{\tau^{\theta,\bar{c}, t}-\tau}{\theta}\right) \ \Rightarrow \ \tau^{\theta,\bar{c}, t}\le \tau+\theta,
\end{split}
\end{equation}
\begin{equation}
\label{absQminQE2}
\left|Q_r^{\theta,\bar{c}, t}-Q_r\right|\le \sup_{r\in[t,t+\theta]}\left|\gamma^{\theta,\bar{c}, t}_r\right|, \qquad \forall r\in[t,T],
\end{equation}
\begin{equation}
\label{QminQE2}
\text{if } \tau^{\theta,\bar{c}, t}<T, \ Q_{\tau}^{\theta,\bar{c}, t}-Q_{\tau^{\theta,\bar{c}, t}}^{\theta,\bar{c}, t}=-\gamma^{\theta,\bar{c}, t}_{t+\theta}, \quad \text{if } \tau^{\theta,\bar{c}, t}=T, \ Q_{\tau}^{\theta,\bar{c}, t}-Q_{\tau^{\theta,\bar{c}, t}}^{\theta,\bar{c}, t}\le -\gamma^{\theta,\bar{c}, t}_{t+\theta}.
\end{equation}

\item \textbf{On the event} $E^{\theta,\bar{c}, t}_{3}$:
\begin{table}[H]
\begin{center}
\begin{tabular}{c|c|c|c|c|c|c|c|c}
\multicolumn{1}{C{2cm}}{}&\multicolumn{2}{C{3cm}}{\kern-1em $t$}&\multicolumn{2}{C{3cm}}{\kern-1em$t+\theta$}&\multicolumn{2}{C{3cm}}{$\tau^{\theta,\bar{c}, t}=\tau$}&\multicolumn{2}{C{3cm}}{\kern-0.5em $T$}\\[-0.5ex]
\multicolumn{2}{R{3cm}|}{$ r\in $}&\multicolumn{2}{C{3cm}|}{}&\multicolumn{2}{C{3cm}|}{ }&\multicolumn{2}{C{3cm}|}{ }&\kern-5em \Arrow{1.5cm}\\ [-1.5ex]\cline{3-8}
\multicolumn{9}{C{15cm}}{ }\\
\multicolumn{2}{C{3.5cm}|}{$c_r^{\theta,\bar{c}, t}=$}&\multicolumn{2}{C{3cm}|}{$\bar{c}$}&\multicolumn{2}{C{3cm}|}{$c_r$}&\multicolumn{2}{C{3cm}|}{$0$}&\multicolumn{1}{C{1.5cm}}{}\\  
\multicolumn{2}{C{3.5cm}|}{$Q_r^{\theta,\bar{c}, t}=$}&\multicolumn{2}{C{3cm}|}{$Q_r-\gamma^{\theta,\bar{c}, t}_r$}&\multicolumn{2}{C{3cm}|}{$Q_r-\gamma^{\theta,\bar{c}, t}_{t+\theta}$}&\multicolumn{2}{C{3cm}|}{$0$}&\multicolumn{1}{C{1.5cm}}{}\\  
\multicolumn{2}{C{3.5cm}|}{$c_r^{\theta,\bar{c}, t}-c_r=$}&\multicolumn{2}{C{3cm}|}{$\bar{c}-c_r$}&\multicolumn{2}{C{3cm}|}{$0$}&\multicolumn{2}{C{3cm}|}{$0$}&\multicolumn{1}{C{1.5cm}}{}\\
\multicolumn{2}{C{3.5cm}|}{$Q_r^{\theta,\bar{c}, t}-Q_r=$}&\multicolumn{2}{C{3cm}|}{$-\gamma^{\theta,\bar{c}, t}_r$}&\multicolumn{2}{C{3cm}|}{$-\gamma^{\theta,\bar{c}, t}_{t+\theta}$}&\multicolumn{2}{C{3cm}|}{$0$}&\multicolumn{1}{C{1.5cm}}{}\\  
\end{tabular}
\end{center}
\end{table}
From previous scheme we conclude that on the event $E^{\theta,\bar{c}, t}_{3}$
\begin{equation}
\label{QtauE3}
\begin{split}
&\text{if } Q_T=0, \ 0\le Q_{\tau^{\theta,\bar{c}, t}}^{\theta,\bar{c}, t}=-\gamma^{\theta,\bar{c}, t}_{t+\theta},\\
&\text{if } \tau=\tau^{\theta,\bar{c}, t}<T, \ 0= Q_{\tau^{\theta,\bar{c}, t}}^{\theta,\bar{c}, t}=-\gamma^{\theta,\bar{c}, t}_{t+\theta},\\
&\text{if } Q_T>0, \ 0\le Q_{\tau^{\theta,\bar{c}, t}}^{\theta,\bar{c}, t}=Q_{T}-\gamma^{\theta,\bar{c}, t}_{t+\theta},
\end{split}
\end{equation}
\begin{equation}
\label{absQminQE3}
\left|Q_r^{\theta,\bar{c}, t}-Q_r\right|\le \sup_{r\in[t,t+\theta]}\left|\gamma^{\theta,\bar{c}, t}_r\right|, \qquad \forall r\in[t,T],
\end{equation}
\begin{equation}
\label{QminQE3}
Q_{\tau}-Q_{\tau^{\theta,\bar{c}, t}}=0.
\end{equation}
\end{enumerate}

From previous schemes we derive the following Lemmas.
\begin{lemma}
\label{lemmaoncthetaadmissible}
Let $t\in[0,\tau)$ be fixed, let $\bar{c}\ge 0$ and let $\theta\in\left(0,(T-t)\wedge\frac{q}{\bar{c}}\right)$. Then the control $c^{\theta,\bar{c}, t}$ in $\eqref{defctheta}$ is admissible. 
\end{lemma}
\begin{proof}
Firstly, we observe that control $c^{\theta,\bar{c}, t}_r$ is non-negative for any $r\in[t,\tau)$. If $\tau^{\theta,\bar{c}, t}> \tau$, i.e. if we are in the event $E^{\theta,\bar{c}, t}_2$, then using $\eqref{gammanegE2}$ we get that $\gamma^{\theta,\bar{c}, t}_{t+\theta}<0$ and so the control $c^{\theta,\bar{c}, t}_r$ is non-negative for any $r\ge \tau$ as well. Progressive measurability, right-continuity and square integrability of $c^{\theta,\bar{c}, t}$ immediately follow.
\end{proof}

\begin{lemma}
\label{lemmaonc}
Let $t\in[0,\tau)$ be fixed, let $\bar{c}\ge 0$ and let $\theta\in\left(0,(T-t)\wedge\frac{q}{\bar{c}}\right)$. Then
\begin{align*}
c_r^{\theta,\bar{c}, t}-c_r&=0, \quad \forall r\in\left[t+\theta, \tau^{\theta,\bar{c}, t}_{\min}\vee (t+\theta)\right],\\
Q_r^{\theta,\bar{c}, t}-Q_r&=-\gamma^{\theta,\bar{c}, t}_{t+\theta}, \quad \forall r\in\left[t+\theta, \tau^{\theta,\bar{c}, t}_{\min}\vee (t+\theta)\right].
\end{align*}
\end{lemma}
\begin{proof}
Looking at schemes on pages \pageref{QtauE1}-\pageref{QtauE3}, it follows that on the event $E_1^{\theta,\bar{c}, t}$, $\tau^{\theta,\bar{c}, t}_{\min}\vee (t+\theta)=\tau^{\theta,\bar{c}, t}$, on the event $E_2^{\theta,\bar{c}, t}\cap \{ \tau>t+\theta \}$, $\tau^{\theta,\bar{c}, t}_{\min}\vee (t+\theta)=\tau$, on the event $E_2^{\theta,\bar{c}, t}\cap \{ \tau\le t+\theta \}$, $\tau^{\theta,\bar{c}, t}_{\min}\vee (t+\theta)=t+\theta$ and on the event $E_3^{\theta,\bar{c}, t}$, $\tau^{\theta,\bar{c}, t}_{\min}\vee (t+\theta)=\tau^{\theta,\bar{c}, t}=\tau$. Then, the result immediately follows.

\end{proof}

\begin{lemma}
\label{lemmaonlimdiffQ}
Let $t\in[0,\tau)$ and $\bar{c}\ge 0$ be fixed. Then
\begin{align}
\label{relonlimsupgamma}
\lim_{\theta\to 0}\E^t\left[\sup_{r\in[t,t+\theta]}\left|\gamma^{\theta,\bar{c}, t}_r\right|\right]=0,\\
\label{relonlimdiffQ}
\lim_{\theta\to 0}\E^t\left[\sup_{r\in[t,T]}\left|Q_r^{\theta,\bar{c}, t}-Q_r\right|\right]=0.
\end{align}
\end{lemma}
\begin{proof}
Let $\theta\in\left(0,(T-t)\wedge\frac{q}{\bar{c}}\right)$ be fixed. From definition of $\gamma^{\theta,\bar{c}, t}_r$ in $\eqref{defgammatheta}$ we immediately see that
\begin{equation}
\label{eqsupgamma}
\sup_{r\in[t,t+\theta]}\left|\gamma^{\theta,\bar{c}, t}_r\right|\le\theta\left(\bar{c}+\sup_{r\in[t,t+\theta]}c_r\right)\quad \text{a.s.}.
\end{equation}
Merging $\eqref{absQminQE1}$, $\eqref{absQminQE2}$ and $\eqref{absQminQE3}$, we see that 
\begin{equation*}
\left|Q_r^{\theta,\bar{c}, t}-Q_r\right|\le \sup_{r\in[t,t+\theta]}\left|\gamma^{\theta,\bar{c}, t}_r\right|, \quad \forall r\in[t, T].
\end{equation*}
Therefore, merging $\eqref{eqsupgamma}$ and previous expression we get that
\begin{equation*}
\E^t\left[\sup_{r\in[t,T]}\left|Q_r^{\theta,\bar{c}, t}-Q_r\right|\right]\le \theta\left(\bar{c}+\E^t\left[\sup_{r\in[t,t+\theta]}c_r\right]\right).
\end{equation*}
We conclude the proof by using $\eqref{asscquad}$.
\end{proof}

\begin{lemma}
\label{limitevents}
Let $\bar{c}\ge 0$ and $t\in[0,\tau)$ be fixed. Then
\begin{align}
\label{eqlimittauless}
\lim_{\theta\to 0}\Pro\left(\{\tau\le t+\theta\}\right)=0,\\
\label{eqlimitevents1}
\lim_{\theta\to 0}\Pro\left(E^{\theta,\bar{c}, t}_{1}\cap \{Q_T>0\}\right)=0,\\
\label{eqlimiteventseps1}
\lim_{\theta\to 0}\Pro\left(E^{\theta,\bar{c}, t}_{1}\cap \{\bar{c}< c_t\}\right)=0,\\
\label{eqlimitevents2}
\lim_{\theta\to 0}\Pro\left(E^{\theta,\bar{c}, t}_{2}\cap \{\tau^{\theta,\bar{c}, t}=T\}\right)=0,\\
\label{eqlimiteventseps3}
\lim_{\theta\to 0}\Pro\left(E^{\theta,\bar{c}, t}_{3}\cap \{\bar{c}>c_t\}\cap \{Q_T=0\}\right)=0,\\
\label{eqlimitevents3}
\lim_{\theta\to 0}\Pro\left(E^{\theta,\bar{c}, t}_{3}\cap\{ \tau<T \}\cap \{ \bar{c}< c_t \}\right)=0.
\end{align}
\end{lemma}
\begin{proof}
We firstly prove $\eqref{eqlimittauless}$. We have that
\begin{equation*}
\begin{split}
\lim_{\theta\to 0}\Pro\left(\{\tau\le t+\theta\}\right)&=\lim_{n\to \infty}\Pro\left(\left\{\tau\le t+\frac{1}{n}\right\}\right)=\Pro\left(\bigcap_{n\ge \bar{n}}\left\{\tau\le t+\frac{1}{n}\right\}\right)= \Pro\left(\{\tau\le t\}\right)=0.
\end{split}
\end{equation*}
In previous calculations we used that for any $n\ge \bar{n}:=\left\lceil\frac{1}{(T-t)\wedge\frac{q}{\bar{c}}}\right\rceil$, the sequence of events $\left\{\tau\le t+\frac{1}{n}\right\}$ is decreasing. This concludes proof of $\eqref{eqlimittauless}$. 

We now prove $\eqref{eqlimitevents1}$. Using definition of $Q$, we have that under event $E_1^{\theta,\bar{c}, t}$, $Q_{\tau}=Q_{\tau^{\theta,\bar{c}, t}}-\int_{\tau^{\theta,\bar{c}, t}}^{\tau} c_r dr\le Q_{\tau^{\theta,\bar{c}, t}}$. Moreover, if $Q_T>0$, then it necessarily implies that $\tau=T$. Using $\eqref{QtauE1}$ we have that
\begin{equation*}
\begin{split}
\lim_{\theta\to 0}\Pro\left(E^{\theta,\bar{c}, t}_{1}\cap \{Q_T>0\}\right)&\le\lim_{\theta\to 0}\Pro\left(\{Q_\tau\le Q_{\tau^{\theta,\bar{c}, t}}=\gamma^{\theta,\bar{c}, t}_{t+\theta}\}\cap \{Q_\tau>0\}\right)\\
&\le\lim_{n\to \infty}\Pro\left( \left\{Q_\tau\le \sup_{r\in\left[t,t+\frac 1n\right]}\left| \gamma^{\frac 1n,\bar{c}, t}_r \right|\right\}\cap \{Q_\tau>0\}\right)\\
&=\Pro\left(\bigcap_{n\ge \bar{n}}\left\{Q_\tau\le \int_t^{t+\frac 1n}\left| \bar{c}-c_r \right|dr\right\}\cap \{Q_\tau>0\}\right)\\
&= \Pro\left(\left\{Q_\tau=0\right\}\cap \{Q_\tau>0\}\right)=0.
\end{split}
\end{equation*}
In previous calculations we used that the sequence of events $\left(\left\{Q_\tau\le \int_t^{t+\frac 1n}\left| \bar{c}-c_r \right|dr\right\}\right)_{n\ge \bar{n}}$ is decreasing and using right-continuity of $c$, $\int_t^{t+\frac 1n}\left| \bar{c}-c_r \right|dr$ converges to $0$ a.s., as $n\to \infty$. This concludes proof of $\eqref{eqlimitevents1}$. 

We now prove $\eqref{eqlimiteventseps1}$. Using $\eqref{QtauE1}$, we get
\begin{equation*}
\begin{split}
\lim_{\theta\to 0}\Pro\left(E^{\theta,\bar{c}, t}_{1}\cap \{\bar{c}< c_t\}\right)&\le\lim_{\theta\to 0}\Pro\left(\left\{\theta\bar{c}-\int_t^{t+\theta}c_s ds>0\right\}\cap \{\bar{c}<c_t\}\right)\\
&\le\lim_{n\to \infty}\Pro\left( \left\{\bar{c}> \inf_{r\in\left[t,t+\frac 1n\right]}c_r\right\}\cap \{\bar{c}<c_t\}\right)\\
&=\Pro\left(\bigcap_{n\ge \bar{n}}\left\{\bar{c}> \inf_{r\in\left[t,t+\frac 1n\right]}c_r\right\}\cap \{\bar{c}<c_t\}\right)= \Pro\left( \{\bar{c}\ge c_t\}\cap \{\bar{c}<c_t\}\right)=0.
\end{split}
\end{equation*}
In previous calculations we used right-continuity of process $c$ and that the sequence of events $\left\{c_t> \inf_{r\in\left[t,t+\frac 1n\right]}c_r\right\}_{n\ge \bar{n}}$ is decreasing. This concludes proof of $\eqref{eqlimiteventseps1}$. 

We now prove $\eqref{eqlimitevents2}$. Using $\eqref{QtauE2}$, we get
\begin{equation*}
\begin{split}
\lim_{\theta\to 0}\Pro\left(E^{\theta,\bar{c}, t}_{2}\cap \{\tau^{\theta,\bar{c}, t}=T\}\right)&\le\lim_{\theta\to 0}\Pro\left(\left\{ \tau+\theta\ge \tau^{\theta,\bar{c}, t} \right\}\cap \{\tau<T\}\cap \{\tau^{\theta,\bar{c}, t}=T\}\right)\\
&=\lim_{n\to \infty}\Pro\left(\left\{\tau+\frac 1n\ge T\right\}\cap\{\tau<T\}\right)\\
&=\Pro\left(\bigcap_{n\ge \bar{n}}\left\{\tau+\frac 1n\ge T\right\}\cap\{\tau<T\}\right)=\Pro\left(\left\{\tau\ge T\right\}\cap\{\tau<T\}\right)=0.
\end{split}
\end{equation*}
In previous calculations we used that the sequence of events $\left\{\tau+\frac 1n\ge T\right\}_{n\ge \bar{n}}$ is decreasing. This concludes proof of $\eqref{eqlimitevents2}$. 

We now prove $\eqref{eqlimiteventseps3}$. Using $\eqref{QtauE3}$, we get
\begin{equation*}
\begin{split}
\lim_{\theta\to 0}\Pro\left(E^{\theta,\bar{c}, t}_{3}\cap \{\bar{c}>c_t\}\cap \{Q_T=0\}\right)&\le\lim_{\theta\to 0}\Pro\left(\{\gamma^{\theta,\bar{c}, t}_{t+\theta}\le 0\}\cap \{\bar{c}>c_t\}\right)\\
&\le\lim_{n\to \infty}\Pro\left( \left\{\bar{c}\le \sup_{r\in\left[t,t+\frac 1n\right]}c_r\right\}\cap \{\bar{c}>c_t\}\right)\\
&=\Pro\left(\bigcap_{n\ge \bar{n}}\left\{\bar{c}\le \sup_{r\in\left[t,t+\frac 1n\right]}c_r\right\}\cap \{\bar{c}>c_t\}\right)\\
&= \Pro\left( \{\bar{c}\le c_t\}\cap \{\bar{c}>c_t\}\right)=0.
\end{split}
\end{equation*}
In previous calculations we used right-continuity of process $c$ and that the sequence of events $\left\{c_t\le \sup_{r\in\left[t,t+\frac 1n\right]}c_r\right\}_{n\ge \bar{n}}$ is decreasing. This concludes the proof of $\eqref{eqlimiteventseps3}$. $\eqref{eqlimitevents3}$ can be proved similarly. 
\end{proof}

\begin{lemma}
\label{lemmaonconvergencetau}
Let $t\in[0,\tau)$ and $\bar{c}\ge 0$ be fixed. Then
\begin{equation}
\label{relonconvergencetauL1}
\lim_{\theta\to 0} \E^t\left[\left|\tau^{\theta,\bar{c}, t}-\tau\right|\right]=0.
\end{equation}
\end{lemma}
\begin{proof}
We firstly prove that $\lim_{\theta\to 0} \tau^{\theta,\bar{c}, t}=\tau$ pointwise $\Pro$-almost everywhere. We assume on the contrary there exists a non-null event $\mathcal{E}$, so that $\lim_{\theta\to 0} \left|\tau^{\theta,\bar{c}, t}-\tau\right|>0 $ on $\mathcal{E}$, which means that
\begin{equation}
\label{eqq899}
\exists \gamma>0 \ \text{ s.t. } \forall \bar{\theta}\in\left(0,(T-t)\wedge\frac{q}{\bar{c}}\wedge \gamma\right), \ \exists \theta\in(0,\bar{\theta}) \text{ s.t. }  \left|\tau-\tau^{\theta,\bar{c}, t}\right|> \gamma \ \text{ on } \mathcal{E}.
\end{equation}
Using that under event $E^{\theta,\bar{c}, t}_{1}$, $\tau>\tau^{\theta,\bar{c}, t}$ and so $\left|\tau-\tau^{\theta,\bar{c}, t}\right|> \gamma$ implies that $\tau-\tau^{\theta,\bar{c}, t}> \gamma$, which implies $Q_{\tau-\gamma}=Q_{\tau^{\theta,\bar{c}, t}}-\int_{\tau^{\theta,\bar{c}, t}}^{\tau-\gamma} c_r dr\le Q_{\tau^{\theta,\bar{c}, t}}=\gamma^{\theta,\bar{c}, t}_{t+\theta}$. Moreover, using that under event $E^{\theta,\bar{c}, t}_{2}$, $\tau^{\theta,\bar{c}, t}=(\tau+\theta)\wedge T$, $\left|\tau-\tau^{\theta,\bar{c}, t}\right|> \gamma$ implies that $\theta\ge (\tau+\theta)\wedge T-\tau> \gamma$, which is never verified, as $\theta<\bar{\theta}<\gamma$. Moreover, under event $E^{\theta,\bar{c}, t}_{3}$, we have that $\tau^{\theta,\bar{c}, t}=\tau$, which never satisfies $\left|\tau-\tau^{\theta,\bar{c}, t}\right|> \gamma$. Therefore, we have that $\eqref{eqq899}$ implies that
\begin{equation}
\label{eqq899bis}
\exists \gamma>0 \ \text{ s.t. } \forall \bar{\theta}\in\left(0,(T-t)\wedge\frac{q}{\bar{c}}\wedge \gamma\right), \ \exists \theta\in(0,\bar{\theta}) \text{ s.t. } Q_{\tau-\gamma}\le \gamma^{\theta,\bar{c}, t}_{t+\theta} \ \text{ on } \mathcal{E}.
\end{equation}
Reminding that $\gamma^{\theta,\bar{c}, t}_{t+\theta} =\int_t^{(t+\theta)\wedge\tau}(\bar{c}-c_r)dr\le \int_t^{(t+\theta)\wedge\tau}\bar{c}dr\le \bar{c}\theta$, expression $\eqref{eqq899bis}$ implies that $Q_{\tau-\gamma}=0$ on $\mathcal{E}$, which contradicts definition of $\tau$, as $\tau$ should be the first time in which $Q_r$ hits $0$. Therefore, we conclude that $\mathcal{E}$ must be a set with $0$ measure, which implies $\Pro^t$-almost everywhere pointwise convergence of $\tau^{\theta,\bar{c}, t}$ to $\tau$. To prove $\eqref{relonconvergencetauL1}$ we observe that $\left|\tau^{\theta,\bar{c}, t}-\tau\right|\le T$, independently on $\theta$. Applying the dominated convergence theorem (DCT) we get $\eqref{relonconvergencetauL1}$.
\end{proof}

\begin{lemma}
Let $t\in[0,\tau)$, $\bar{c}\ge 0$ and $p\in[1,2)$ be fixed. Then
\begin{align}
\label{relonlimgamma}
\lim_{\theta\to 0}\E^t\left[\left|\bar{c}-c_t-\frac{\gamma^{\theta,\bar{c}, t}_{t+\theta}}{\theta}\right|^p\right]=0,\\
\label{lemmaonqpos}
\lim_{\theta\to 0}\E^t\left[\left|\frac{Q_{\tau^{\theta,\bar{c}, t}_{\min}}^{\theta,\bar{c}, t}-Q_{\tau^{\theta,\bar{c}, t}_{\min}}}{\theta}+\bar{c}-c_t\right|^p\right]=0.
\end{align}
\end{lemma}
\begin{proof}
Let $t\in[0,\tau)$ and $\theta\in\left(0,(T-t)\wedge\frac{q}{\bar{c}}\right)$ be fixed. We firstly observe that
\begin{equation*}
\begin{split}
\left|\frac{\gamma^{\theta,\bar{c}, t}_{t+\theta}}{\theta}\right|\le \frac 1\theta \int_t^{t+\theta} |\bar{c}-c_s|ds\le |\bar{c}|+\sup_{s\in[t,T]}|c_s|,
\end{split}
\end{equation*}
which is $L^p$-integrable thanks to assumption $\eqref{asscquad}$. Moreover, we have that
\begin{equation*}
\begin{split}
\left|\bar{c}-c_t-\frac{\gamma^{\theta,\bar{c}, t}_{t+\theta}}{\theta}\right|=\left|\bar{c}-c_t-\frac 1\theta \int_t^{t+\theta} (\bar{c}-c_s)ds\right|.
\end{split}
\end{equation*}
Therefore, by using right-continuity of control $c$ and mean-value theorem, we conclude that the pointwise limit of the expression inside the expectation in $\eqref{relonlimgamma}$ is $0$. Finally, by using DCT we conclude the proof of $\eqref{relonlimgamma}$.

Moreover, looking at schemes in pages \pageref{QtauE1}-\pageref{QtauE3}, we can immediately see that
\begin{equation*}
Q_{\tau^{\theta,\bar{c}, t}_{\min}}^{\theta,\bar{c}, t}-Q_{\tau^{\theta,\bar{c}, t}_{\min}}=-\gamma^{\theta,\bar{c}, t}_{t+\theta}.
\end{equation*}
Therefore, by applying $\eqref{relonlimgamma}$ and previous expression into $\eqref{lemmaonqpos}$ we prove the Lemma.
\end{proof}

\begin{lemma}
\label{lemma5}
Let $t\in[0,\tau)$, $\bar{c}\ge 0$ and $p\in[1,2)$ be fixed. Then
\begin{equation*}
\lim_{\theta\to 0}\E^t\left[\left|\frac{\hat{Q}^{\theta,\bar{c}, t}_{\tau^{\theta,\bar{c}, t}}-\hat{Q}^{\theta,\bar{c}, t}_\tau}{\theta}-(\bar{c}-c_t)\mathds{1}_{\Lambda(t,\bar{c})}\right|^p\right]=0.
\end{equation*}
\end{lemma}
\begin{proof}
Let $t\in[0,\tau)$ and $\theta\in\left(0,(T-t)\wedge\frac{q}{\bar{c}}\right)$ be fixed. Using schemes in pages \pageref{QtauE1}-\pageref{QtauE3}, $\eqref{QtauE1}$, $\eqref{QminQE1}$, $\eqref{QtauE2}$, $\eqref{QminQE2}$, using H\"older's inequality (with coefficients $\frac{p+2}{2p}$ and $\frac{p+2}{2-p}$), reminding that $E^{\theta,\bar{c}, t}_{2}$ implies that $Q_T=0$, we get
\begin{align*}
\E^t&\left[\left|\frac{\hat{Q}^{\theta,\bar{c}, t}_{\tau^{\theta,\bar{c}, t}}-\hat{Q}^{\theta,\bar{c}, t}_\tau}{\theta}-(\bar{c}-c_t)\mathds{1}_{\Lambda(t,\bar{c})}\right|^p\right]=\E^t\left[\left|\frac{Q_{\tau^{\theta,\bar{c}, t}}-Q_\tau}{\theta}-(\bar{c}-c_t)\mathds{1}_{\Lambda(t,\bar{c})}\right|^p\mathds{1}_{E^{\theta,\bar{c}, t}_{1}}\right]\\
&\quad+\E^t\left[\left|\frac{Q^{\theta,\bar{c}, t}_{\tau^{\theta,\bar{c}, t}}-Q^{\theta,\bar{c}, t}_{\tau}}{\theta}-(\bar{c}-c_t)\right|^p\mathds{1}_{E^{\theta,\bar{c}, t}_{2}}\right]+\E^t\left[\left|-(\bar{c}-c_t)\mathds{1}_{\Lambda(t,\bar{c})}\right|^p\mathds{1}_{E^{\theta,\bar{c}, t}_{3}}\right]\\
& \le \E^t\left[\left|\frac{Q_{\tau^{\theta,\bar{c}, t}}-Q_\tau}{\theta}\right|^p\mathds{1}_{E^{\theta,\bar{c}, t}_{1}\cap \{ Q_T>0 \}}\right]+\E^t\left[\left|\frac{Q_{\tau^{\theta,\bar{c}, t}}-Q_\tau}{\theta}-(\bar{c}-c_t)\mathds{1}_{\tau<T}\right|^p\mathds{1}_{E^{\theta,\bar{c}, t}_{1}\cap \{ Q_T=0 \}\cap \{ \bar{c}< c_t \}}\right]\\
&\quad +\E^t\left[\left|\frac{\gamma^{\theta,\bar{c}, t}_{t+\theta}}{\theta}-(\bar{c}-c_t)\right|^p\mathds{1}_{E^{\theta,\bar{c}, t}_{1}\cap \{ Q_T=0 \}\cap \{ \bar{c}\ge c_t \}}\right]+\E^t\left[\left|\frac{\gamma^{\theta,\bar{c}, t}_{t+\theta}}{\theta}-(\bar{c}-c_t)\right|^p\mathds{1}_{E^{\theta,\bar{c}, t}_{2}\cap \{ \tau^{\theta,\bar{c}, t}<T \}}\right]\\
&\quad +\E^t\left[\left|\frac{Q^{\theta,\bar{c}, t}_{\tau^{\theta,\bar{c}, t}}-Q^{\theta,\bar{c}, t}_{\tau}}{\theta}-(\bar{c}-c_t)\right|^p\mathds{1}_{E^{\theta,\bar{c}, t}_{2}\cap \{ \tau^{\theta,\bar{c}, t}=T \}}\right]\\
&\quad +\E^t\left[\left|-(\bar{c}-c_t)\mathds{1}_{\{ Q_T=0 \}\cap \{ \bar{c}\ge c_t \}}\right|^p\mathds{1}_{E^{\theta,\bar{c}, t}_{3}}\right]+\E^t\left[\left|-(\bar{c}-c_t)\mathds{1}_{\{ \tau<T \}\cap \{ \bar{c}< c_t \}}\right|^p\mathds{1}_{E^{\theta,\bar{c}, t}_{3}}\right]\\
&\le \left(\E^t\left[\left|\frac{\sup_{r\in[t,t+\theta]}\gamma^{\theta,\bar{c}, t}_r}{\theta}\right|^{\frac{p+2}{2}}\right]\right)^{\frac{2p}{p+2}}\bigg(\Pro^t\left(E^{\theta,\bar{c}, t}_{1}\cap \{ Q_T>0 \}\right)^{\frac{2-p}{p+2}}\\
&\qquad+2^{p-1}\Pro^t\left(E^{\theta,\bar{c}, t}_{1}\cap \{ Q_T=0 \}\cap \{ \bar{c}< c_t \}\right)^{\frac{2-p}{p+2}}+2^{p-1}\Pro^t\left(E^{\theta,\bar{c}, t}_{2}\cap \{ \tau^{\theta,\bar{c}, t}=T \}\right)^{\frac{2-p}{p+2}}\bigg)\\
&\qquad+\E^t\left[\left|\frac{\gamma^{\theta,\bar{c}, t}_{t+\theta}}{\theta}-(\bar{c}-c_t)\right|^p\right]\\
&\quad+|\bar{c}-c_t|^p\bigg(2^{p-1}\Pro^t\left( E^{\theta,\bar{c}, t}_{1}\cap \{ Q_T=0 \}\cap \{ \bar{c}< c_t \} \right)+2^{p-1}\Pro^t\left( E^{\theta,\bar{c}, t}_{2}\cap \{ \tau^{\theta,\bar{c}, t}=T \} \right)\\
&\qquad\qquad\qquad+\Pro^t\left(E^{\theta,\bar{c}, t}_{3}\cap\{ Q_T=0 \}\cap \{ \bar{c}> c_t \}\right)+\Pro^t\left(E^{\theta,\bar{c}, t}_{3}\cap\{ \tau<T \}\cap \{ \bar{c}< c_t \}\right)\bigg).
\end{align*}
Taking the limit of the above expression and using $\eqref{eqsupgamma}$, $\eqref{eqlimitevents1}$, $\eqref{eqlimiteventseps1}$, $\eqref{eqlimitevents2}$, $\eqref{eqlimiteventseps3}$, $\eqref{eqlimitevents3}$ and $\eqref{relonlimgamma}$ we conclude the proof of the Lemma.
\end{proof}

\begin{lemma}
\label{lemmaconvQhat}
Let $t\in[0,\tau)$ and $\bar{c}\ge 0$ be fixed. Then,
\begin{align}
\label{eqconvQhat}
\lim_{\theta\to 0} \E^t\left[\left| \hat{Q}^{\theta,\bar{c}, t}_{\tau^{\theta,\bar{c}, t}}-\hat{Q}^{\theta,\bar{c}, t}_\tau \right|\right]=0,\\
\label{eqconvQhatsingle}
\lim_{\theta\to 0}\E^t\left[ \left|\hat{Q}^{\theta,\bar{c}, t}_\tau-Q_\tau\right|\right]=0.
\end{align}
\end{lemma}
\begin{proof}
Let $\theta\in\left(0,(T-t)\wedge\frac{q}{\bar{c}}\right)$ be fixed. By using $\eqref{QminQE1}$, $\eqref{QminQE2}$ and $\eqref{QminQE3}$, we have that
\begin{equation*}
\begin{split}
\E^t\left[\left| \hat{Q}^{\theta,\bar{c}, t}_{\tau^{\theta,\bar{c}, t}}-\hat{Q}^{\theta,\bar{c}, t}_\tau \right|\right]&\le \E^t\left[\left| Q_{\tau^{\theta,\bar{c}, t}}-Q_\tau \right|\mathds{1}_{E^{\theta,\bar{c}, t}_1}\right]+\E^t\left[\left| Q^{\theta,\bar{c}, t}_{\tau^{\theta,\bar{c}, t}}-Q^{\theta,\bar{c}, t}_\tau \right|\mathds{1}_{E^{\theta,\bar{c}, t}_2}\right]\\
&\le \E^t\left[\sup_{r\in[t,t+\theta]}\left|\gamma^{\theta,\bar{c}, t}_r\right|\right].
\end{split}
\end{equation*}
Taking the limit of the above expression and using $\eqref{relonlimsupgamma}$ we conclude the proof of $\eqref{eqconvQhat}$. We now prove $\eqref{eqconvQhatsingle}$. By $\eqref{absQminQE2}$ and $\eqref{absQminQE3}$, we have that
\begin{equation*}
\begin{split}
\E^t\left[\left| \hat{Q}^{\theta,\bar{c}, t}_\tau-Q_\tau \right|\right]&\le \E^t\left[\left| Q^{\theta,\bar{c}, t}_{\tau}-Q_\tau \right|\mathds{1}_{E^{\theta,\bar{c}, t}_2}\right]+\E^t\left[\left| Q^{\theta,\bar{c}, t}_\tau-Q_\tau \right|\mathds{1}_{E^{\theta,\bar{c}, t}_3}\right]\\
&\le \E^t\left[\sup_{r\in[t,t+\theta]}\left|\gamma^{\theta,\bar{c}, t}_r\right|\right].
\end{split}
\end{equation*}
Taking the limit of the above expression and using $\eqref{relonlimsupgamma}$ we conclude the proof of $\eqref{eqconvQhatsingle}$.
\end{proof}

\begin{lemma}
\label{lemma4}
For any $(x,q),(x,q')\in\mathcal{O}$ with $q\ne q'$, we have that
\begin{equation*}
\left| \frac{g(x,q)-g(x,q')}{q-q'}-\partial_q g\left(x,q'\right)\right|\le  K |q-q'|.
\end{equation*}
\end{lemma}
\begin{proof}
We observe that
\begin{equation*}
\begin{split}
\frac{g(x,q)-g(x,q')}{q-q'}=\int_0^1\partial_q g\left(x, q'+\lambda(q-q')\right)\; d\lambda
\end{split}
\end{equation*}
and so using Lipschitz continuity of $\partial_q g$ in Assumption \ref{assumptions}, we get
\begin{equation*}
\begin{split}
\left| \frac{g(x,q)-g(x,q')}{q-q'}-\partial_q g\left(x,q'\right)\right|&=\left| \int_0^1\partial_q g\left(x, q'+\lambda(q-q')\right)\; d\lambda-\partial_q g\left(x,q'\right)\right|\\
&\le \int_0^1\left|\partial_q g\left(x, q'+\lambda(q-q')\right)-\partial_q g\left(x,q'\right)\right|d\lambda\le \frac{K}{2}|q-q'|.
\end{split}
\end{equation*}
This proves the lemma.
\end{proof}

We introduce a process used in the proof of Theorem \ref{SMPtheorem}. 
  Define $(\xi_r)_{r\in[t,T]}$ as
\begin{equation}
\label{defxi}
\xi_r := f(t, \bar{c}, X_{t}, Q_{t})-f(t, c_{t}, X_{t}, Q_{t})  - \int_t^r (\bar{c}-c_t) \partial_q f(s, c_s,X_s,Q_s)ds.
\end{equation}
 $\xi$ is the same as process $\zeta$ in Bensoussan~\cite{Bensoussan2006}. The corresponding part of $z$ in Bensoussan~\cite{Bensoussan2006} for process $Q$ would be constantly equal to $\bar{c}-c_t$, as it can be inferred with a simple calculus.

Our proof relies on the arguments in Bensoussan~\cite{Bensoussan2006} with one key difference. Due to the presence of the control-dependent stopping time $\tau$ in our setting, it makes necessary the introduction of the stopping time $\tau^{\theta,\bar{c}, t}$ as well. This complicates all the proofs and makes necessary many adjustments, especially on those results in Bensoussan~\cite{Bensoussan2006} that concern terminal time $T$ that must be adapted to $\tau$ or $\tau^{\theta,\bar{c}, t}$ accordingly.

\begin{lemma}
\label{lemmaconvrvtheta}
Let $t\in[0,\tau)$ and $\bar{c}\ge 0$ be fixed. Then,
\begin{align*}
\lim_{\theta\to 0} \E^t\left[\left|\xi_{\tau^{\theta,\bar{c}, t}}-\xi_{\tau}\right|\right]=0,\\
\lim_{\theta\to 0} \E^t\left[\left|Q_{\tau^{\theta,\bar{c}, t}}-Q_{\tau}\right|\right]=0.
\end{align*}
\end{lemma}
\begin{proof}
From $\eqref{defxi}$ and using boundedness of $\partial_q f$, we have that
\begin{equation*}
\E^t\left[\left|\xi_{\tau^{\theta,\bar{c}, t}}-\xi_{\tau}\right|\right]=\E^t\left[\left|-\int_{\tau}^{\tau^{\theta,\bar{c}, t}}(\bar{c}-c_t) \partial_q f(r, c_r,X_r,Q_r)dr\right|\right]\le K|\bar{c}-c_t|\E^t\left[\left|\tau^{\theta,\bar{c}, t}-\tau\right|\right].
\end{equation*}
Therefore, by taking the limit of previous expression and applying $\eqref{relonconvergencetauL1}$, we get the first limit in the statement. Moreover, from definition of $Q_r$, we have
\begin{equation*}
\E^t\left[\left|Q_{\tau^{\theta,\bar{c}, t}}-Q_{\tau}\right|\right]=\E^t\left[\left|-\int_{\tau}^{\tau^{\theta,\bar{c}, t}}c_r dr\right|\right]\le \left(\E^t\left[\left|\int_t^T c_r^2 dr\right|\right]\right)^{1/2}\left(\E^t\left[\left|\tau^{\theta,\bar{c}, t}-\tau\right|\right]\right)^{1/2}.
\end{equation*}
Therefore, by taking the limit of previous expression, using $\eqref{asscquad}$ and applying $\eqref{relonconvergencetauL1}$, we finish the proof of the Lemma.
\end{proof}

\begin{lemma}
\label{lemma3onf}
Let $t\in[0,\tau)$ and $\bar{c}\ge 0$ be fixed. Then
\begin{equation*}
\lim_{\theta\to 0}\E^t\left[\frac {1}{\theta}\int_t^{\tau^{\theta,\bar{c}, t}_{\min}} \left(f\left(s,c_s^{\theta,\bar{c}, t}, X_s,Q_s^{\theta,\bar{c}, t}\right)-f(s,c_s,X_s,Q_s)\right)ds-\xi_{\tau^{\theta,\bar{c}, t}_{\min}}\right]=0.
\end{equation*}
\end{lemma}
\begin{proof}
Let $\theta\in\left(0,(T-t)\wedge\frac{q}{\bar{c}}\right)$ be fixed. We denote for any $r\in[t,T]$
\begin{equation*}
\tilde{f}_r^{\theta,\bar{c}, t}:=\frac{1}{\theta}\int_t^r\left(f\left(s,c_s^{\theta,\bar{c}, t},X_s,Q_s^{\theta,\bar{c}, t}\right) - f(s,c_s,X_s,Q_s)\right)ds -\xi_r.
\end{equation*}
The proof of this lemma will be divided in 3 steps. In step 1 we prove that
\begin{equation*}
\lim_{\theta\to 0}\E^t\left[\left|\tilde{f}_{t+\theta}^{\theta,\bar{c}, t}\right|\right]=0.
\end{equation*}
In Step 2 we prove that,
\begin{equation*}
\lim_{\theta\to 0}\E^t\left[\tilde{f}_{\tau^{\theta,\bar{c}, t}_{\min}}^{\theta,\bar{c}, t}\mathds{1}_{\tau^{\theta,\bar{c}, t}_{\min}\le t+\theta}\right]=0.
\end{equation*}
In Step 3 we prove that
\begin{equation*}
\lim_{\theta\to 0}\E^t\left[\tilde{f}_{\tau^{\theta,\bar{c}, t}_{\min}}^{\theta,\bar{c}, t}\mathds{1}_{\tau^{\theta,\bar{c}, t}_{\min}> t+\theta}\right]=0.
\end{equation*}
In Bensoussan~\cite{Bensoussan2006} there is no difference between Steps 2 and 3, while in our case we need to consider them both. However, the structure of our proof resembles the one in the reference. Once the proof of the 3 steps is completed, we conclude the proof of the Lemma as follows. By merging the limits above in Steps 2 and 3, we have
\begin{equation*}
\lim_{\theta\to 0}\E^t\left[\tilde{f}_{\tau^{\theta,\bar{c}, t}_{\min}}^{\theta,\bar{c}, t}\right]=\lim_{\theta\to 0}\E^t\left[\tilde{f}_{\tau^{\theta,\bar{c}, t}_{\min}}^{\theta,\bar{c}, t}\mathds{1}_{\tau^{\theta,\bar{c}, t}_{\min}\le t+\theta}\right]+\lim_{\theta\to 0}\E^t\left[\tilde{f}_{\tau^{\theta,\bar{c}, t}_{\min}}^{\theta,\bar{c}, t}\mathds{1}_{\tau^{\theta,\bar{c}, t}_{\min}> t+\theta}\right] =0.
\end{equation*}

\textbf{Step 1.} 
From $\eqref{defxi}$, reminding that $c_t^{\theta,\bar{c}, t}=\bar{c}$ and definition of $\tilde{f}$ we have that for any $r\in[t,t+\theta]$,
\begin{equation}
\label{eqq979}
\begin{split}
\tilde{f}_r^{\theta,\bar{c}, t}&=\frac{1}{\theta}\int_t^r\left(f\left(s,c_s^{\theta,\bar{c}, t}, X_s, Q_s^{\theta,\bar{c}, t}\right)-f(s,c_s^{\theta,\bar{c}, t},X_s,Q_s)\right)ds\\
&\qquad +\frac{1}{\theta}\int_t^r\left(f\left(s,c_s^{\theta,\bar{c}, t}, X_s,Q_s\right)-f(t,c_t^{\theta,\bar{c}, t},X_t,Q_t)\right)ds\\
&\qquad +\frac{1}{\theta}\int_t^r\left(f(t,c_t,X_t,Q_t)-f\left(s,c_s, X_s,Q_s\right)\right)ds\\
&\qquad +f(t,c_t,X_t,Q_t)\left(1-\frac{r-t}{\theta}\right)+f(t,c_t^{\theta,\bar{c}, t},X_t,Q_t)\left(\frac{r-t}{\theta}-1\right)\\
&\qquad+(\bar{c}-c_t)\int_t^r\partial_qf(s,c_s,X_s,Q_s)ds.
\end{split}
\end{equation}
By taking $r=t+\theta$ in previous expression, so that the second last line disappears and using Assumption \ref{secassumptions}, boundedness of $\partial_q f$ and H\"older's inequality, we get
\begin{equation}
\label{eqq57quadris}
\begin{split}
\E^t\left[\left|\tilde{f}_{t+\theta}^{\theta,\bar{c}, t}\right|\right]&\le K\Bigg(\E^t\left[\sup_{r\in[t,t+\theta]}\left|Q_r^{\theta,\bar{c}, t}-Q_r\right|\right]+2\E^t\left[\sup_{r\in[t,t+\theta]}\left|Q_r-Q_t\right|\right]+\theta|\bar{c}-c_t|\\
&\qquad +\left(\left(\E^t\left[\sup_{r\in[t,t+\theta]}\left|X_r-X_t\right|^2\right]\right)^{1/2}+\left(\E^t\left[\sup_{r\in[t,t+\theta]}\left|c_r^{\theta,\bar{c}, t}-c_t^{\theta,\bar{c}, t}\right|^2\right]\right)^{1/2}\right)\cdot\\
&\qquad\quad\cdot\left(\E^t\left[\sup_{r\in[t,t+\theta]}\left(1+2|X_r|+2|c_r^{\theta,\bar{c}, t}|\right)^2\right]\right)^{1/2} + \frac{2}{\theta}\int_t^{t+\theta} |s-t|ds\\
&\qquad +\left(\left(\E^t\left[\sup_{r\in[t,t+\theta]}\left|X_r-X_t\right|^2\right]\right)^{1/2}+\left(\E^t\left[\sup_{r\in[t,t+\theta]}\left|c_r-c_t\right|^2\right]\right)^{1/2}\right)\cdot\\
&\qquad\quad\cdot\left(\E^t\left[\sup_{r\in[t,t+\theta]}\left(1+2|X_r|+2|c_r|\right)^2\right]\right)^{1/2}\Bigg).
\end{split}
\end{equation}
Using DCT, $\E^t\left[\sup_{r\in[t,t+\theta]}\left|c_r-c_t\right|^2\right]$ and $\E^t\left[\sup_{r\in[t,t+\theta]}\left|c_r^{\theta,\bar{c}, t}-c_t^{\theta,\bar{c}, t}\right|^2\right]$ converge to $0$ as $\theta\to 0$. Indeed, $c$ and $c^{\theta,\bar{c}, t}$ are right-continuous and thanks to $\eqref{asscquad}$ and $\eqref{eqsupgamma}$, the arguments of the expectations converge to $0$ a.s. and they are bounded by $2\sup_{r\in[t,T]}|c_r|^2$ and $2\sup_{r\in[t,T]}|c_r^{\theta,\bar{c}, t}|^2$, which are $L^1$-integrable processes. $\E^t\left[\sup_{r\in[t,t+\theta]}\left|X_r-X_t\right|^2\right]$ converges to $0$ using standard arguments in SDE theory (c.f. Krylov~\cite[Corollary 2.5.12]{KrylovN.V1980Cdp}). Moreover, using $L^2$-integrability of $c$ and $c^{\theta,\bar{c}, t}$ and standard arguments in SDE theory (c.f. Krylov~\cite[Corollary 2.5.12]{KrylovN.V1980Cdp}), we get that $\E^t\left[\sup_{r\in[t,t+\theta]}\left(2|X_r|+2|c_r^{\theta,\bar{c}, t}|\right)^2\right]$ and \\ $\E^t\left[\sup_{r\in[t,t+\theta]}\left(2|X_r|+2|c_r|\right)^2\right]$ are bounded independently of $\theta$. Moreover, by definition of $Q_r$,
\begin{equation*}
\E^t\left[\sup_{r\in[t,t+\theta]}\left|Q_r-Q_t\right|\right]=\E^t\left[\int_t^{t+\theta}\left|c_r\right| dr\right]\le\sqrt{\theta}\left(\E^t\left[\int_t^{T}c_r^2 dr\right]\right)^{1/2},
\end{equation*}
which converges to $0$, thanks to $\eqref{asscquad}$. Using $\eqref{relonlimdiffQ}$, we have that $\E^t\left[\left(\sup_{r\in[t,t+\theta]}\left|Q_r^{\theta,\bar{c}, t}-Q_r\right|\right)\right]$ converges to $0$.  Finally, we observe that $\frac{2}{\theta}\int_t^{t+\theta} |s-t|ds=\theta$. Therefore, by taking limit of $\eqref{eqq57quadris}$ we conclude the proof of Step 1.

\textbf{Step 2.} From $\eqref{eqq979}$, using Assumption \ref{secassumptions} and boundedness of $\partial_q f$, we get  a similar expression to $\eqref{eqq57quadris}$. Step 2 can be proved in a similar way to Step 1, the main difference is the term $\E^t\left[\left|\frac{\tau^{\theta,\bar{c}, t}_{\min}-t}{\theta}-1\right|\mathds{1}_{\tau^{\theta,\bar{c}, t}_{\min}\le t+\theta}\right]$. 
 However, by reminding that by $\eqref{eqtauthetabig}$, $\left\{ \tau^{\theta,\bar{c}, t}_{\min}\le t+\theta \right\}=\left\{ \tau\le t+\theta \right\}$, that under event $\tau^{\theta,\bar{c}, t}_{\min}\le t+\theta$, then $\left|\frac{\tau^{\theta,\bar{c}, t}_{\min}-t}{\theta}-1\right|\le 1$ and by using $\eqref{eqlimittauless}$, we conclude that
\begin{equation*}
\begin{split}
\lim_{\theta\to 0}\E^t\left[\tilde{f}_{\tau^{\theta,\bar{c}, t}_{\min}}^{\theta,\bar{c}, t}\mathds{1}_{\tau^{\theta,\bar{c}, t}_{\min}\le t+\theta}\right] &\le \left(\left|f(t,\bar{c},X_t,Q_t)\right|+\left|f(t,c_t,X_t,Q_t)\right|\right)\lim_{\theta\to 0} \Pro\left(\left\{ \tau\le t+\theta \right\}\right)=0.
\end{split}
\end{equation*}
This concludes the proof of Step 2.

\textbf{Step 3.} From $\eqref{defxi}$ and definition of $\tilde{f}$ we have that for any $r\in[t+\theta, T]$,
\begin{align*}
\tilde{f}_r^{\theta,\bar{c}, t}&=\tilde{f}_{t+\theta}^{\theta,\bar{c}, t}+\frac{1}{\theta}\int_{t+\theta}^r\left(f\left(s,c_s^{\theta,\bar{c}, t}, X_s, Q_s^{\theta,\bar{c}, t}\right)-f(s,c_s,X_s,Q_s^{\theta,\bar{c}, t})\right)ds\\
&\qquad+\int_{t+\theta}^r\int_0^1\left(\bar{c}-c_t-\frac{Q_s-Q_s^{\theta,\bar{c}, t}}{\theta}\right)\partial_q f\left(s, c_s, X_s, Q_s+\lambda\left(Q_s^{\theta,\bar{c}, t}-Q_s\right)\right)d\lambda ds\\
&\qquad+\int_{t+\theta}^r\int_0^1(\bar{c}-c_t)\left(\partial_q f(s,c_s, X_s, Q_s)-\partial_q f\left(s, c_s, X_s, Q_s+\lambda\left(Q_s^{\theta,\bar{c}, t}-Q_s\right)\right)\right)d\lambda ds.
\end{align*}
Therefore, by applying Assumption \ref{secassumptions}, then boundedness and Lipschitz continuity of $\partial_q f$ follows, we have that
\begin{equation*}
\left|\tilde{f}_r^{\theta,\bar{c}, t}\right|\le\left|\tilde{f}_{t+\theta}^{\theta,\bar{c}, t}\right|+\frac{K}{\theta}\int_{t+\theta}^r\left|c_s^{\theta,\bar{c}, t}-c_s\right|ds+K\int_{t+\theta}^r\left|\frac{Q_s-Q_s^{\theta,\bar{c}, t}}{\theta}-(\bar{c}-c_t)\right| ds+\frac{K|\bar{c}-c_t|}{2}\int_{t+\theta}^r\left|Q_s^{\theta,\bar{c}, t}-Q_s\right| ds.
\end{equation*}
Therefore, from previous expression and using Lemma \ref{lemmaonc}, we get that
\begin{equation*}
\begin{split}
\left|\E^t\left[\tilde{f}_{\tau^{\theta,\bar{c}, t}_{\min}}^{\theta,\bar{c}, t}\mathds{1}_{\tau^{\theta,\bar{c}, t}_{\min}>t+\theta}\right]\right|&\le \E^t\left[\left|\tilde{f}_{t+\theta}^{\theta,\bar{c}, t}\right|\right]+K\E^t\left[\frac{\mathds{1}_{\tau^{\theta,\bar{c}, t}_{\min}> t+\theta}}{\theta}\int_{t+\theta}^{\tau^{\theta,\bar{c}, t}_{\min}}\left|c_s^{\theta,\bar{c}, t}-c_s\right| ds\right]\\
&\qquad+K\E^t\left[\mathds{1}_{\tau^{\theta,\bar{c}, t}_{\min}> t+\theta}\int_{t+\theta}^{\tau^{\theta,\bar{c}, t}_{\min}}\left|\frac{Q_s-Q_s^{\theta,\bar{c}, t}}{\theta}-(\bar{c}-c_t)\right| ds\right]\\
&\qquad+K\frac{|\bar{c}-c_t|}{2}\E^t\left[\mathds{1}_{\tau^{\theta,\bar{c}, t}_{\min}> t+\theta}\int_{t+\theta}^{\tau^{\theta,\bar{c}, t}_{\min}}\left|Q_s^{\theta,\bar{c}, t}-Q_s\right|ds\right]\\
&\le \E^t\left[\left|\tilde{f}_{t+\theta}^{\theta,\bar{c}, t}\right|\right]+KT\E^t\left[\left|\frac{\gamma^{\theta,\bar{c}, t}_{t+\theta}}{\theta}-(\bar{c}-c_t)\right|\right]+KT\frac{|\bar{c}-c_t|}{2}\E^t\left[\left|\gamma^{\theta,\bar{c}, t}_{t+\theta}\right|\right].
\end{split}
\end{equation*}
By taking limit of the above expression for $\theta\to 0$, by using $\eqref{relonlimsupgamma}$ and $\eqref{relonlimgamma}$ together with Step 1, we conclude the proof of Step 3 and the proof of the Lemma as well.
\end{proof}
\textbf{Proof of Theorem \ref{SMPtheorem}.} Let $t\in[0,\tau)$ be fixed. Since control $c$ is optimal, it necessarily follows that for any $\bar{c}\ge 0$ and for any $\theta>0$
\begin{equation*}
v^{c^{\theta,\bar{c}, t}}(t, x,q)\le v^c(t,x,q),
\end{equation*}
where
\begin{equation*}		
v^\pi(t,x,q) =\E^t\left[ g( X_{\tau^\pi}, Q_{\tau^\pi}^\pi)+ \int_t^{\tau^\pi} f(r, \pi_r, X_r, Q_r^\pi) \;dr\right]		
\end{equation*}		
and $\tau^\pi=T\wedge \min\{r\ge t\; | \; Q_r^\pi=0\}$. We write $Q$ as $Q^c$ and $\tau$ as $\tau^c$ as in Theorem \ref{SMPtheorem}.

Therefore, we necessarily have that for any $\bar{c}\ge 0$
\begin{equation}
\label{negativeJ}
\lim_{\theta\to 0} \frac{v^{c^{\theta,\bar{c}, t}}(t, x,q)-v^c(t,x,q)}{\theta}\le 0,
\end{equation}
provided the limit exists.
By definition of $v^\pi$, reminding that when $\tau^{\theta,\bar{c}, t}_{\min}=\tau^{\theta,\bar{c}, t}$, then for $r\ge \tau^{\theta,\bar{c}, t}$, $\hat{Q}^{\theta,\bar{c}, t}_r=Q_r$ and $\hat{c}^{\theta,\bar{c}, t}_r=c_r$ and when $\tau^{\theta,\bar{c}, t}_{\min}=\tau$, then for $r\ge \tau$, $\hat{Q}^{\theta,\bar{c}, t}_r=Q_r^{\theta,\bar{c}, t}$ and $\hat{c}^{\theta,\bar{c}, t}_r=c_r^{\theta,\bar{c}, t}$
\begin{equation}
\label{derivationepsposfirstpart}
\begin{split}
\lim_{\theta\to 0} \frac{v^{c^{\theta,\bar{c}, t}}(t, x,q)-v^c(t,x,q)}{\theta}&=\lim_{\theta\to 0} \E^t\left[ \frac{g(X_{\tau^{\theta,\bar{c}, t}}, Q^{\theta,\bar{c}, t}_{\tau^{\theta,\bar{c}, t}})-g(X_\tau, Q_\tau)}{\theta}\right]\\
&\qquad+\lim_{\theta\to 0} \E^t\left[ \frac{1}{\theta}\int_{t}^{\tau^{\theta,\bar{c}, t}_{\min}}\left(f(r, c^{\theta,\bar{c}, t}_r, X_r,Q^{\theta,\bar{c}, t}_t)-f(r, c_r, X_r,Q_r)\right) dr\right]\\
&\qquad+\lim_{\theta\to 0} \E^t\left[-\frac{\sign(\tau-\tau^{\theta,\bar{c}, t})}{\theta}\int_{\tau^{\theta,\bar{c}, t}_{\min}}^{\tau^{\theta,\bar{c}, t}_{\max}}f(r, \hat{c}^{\theta,\bar{c}, t}_r, X_r,\hat{Q}^{\theta,\bar{c}, t}_r) dr\right].
\end{split}
\end{equation}
The first line on the right-hand side of $\eqref{derivationepsposfirstpart}$ can be written as
\begin{equation}
\label{derivationepspos}
\begin{split}
&\lim_{\theta\to 0} \E^t\left[ \frac{g(X_{\tau^{\theta,\bar{c}, t}}, Q^{\theta,\bar{c}, t}_{\tau^{\theta,\bar{c}, t}})-g(X_\tau, Q_\tau)}{\theta}\right]=\lim_{\theta\to 0} \E^t\left[ \frac{g(X_{\tau^{\theta,\bar{c}, t}}, Q^{\theta,\bar{c}, t}_{\tau^{\theta,\bar{c}, t}})-g(X_\tau, Q^{\theta,\bar{c}, t}_{\tau^{\theta,\bar{c}, t}})}{\theta}\right]\\
&+\lim_{\theta\to 0} \E^t\left[ \frac{g(X_\tau, Q^{\theta,\bar{c}, t}_{\tau^{\theta,\bar{c}, t}_{\min}})-g(X_\tau, Q_{\tau^{\theta,\bar{c}, t}_{\min}})}{\theta}\right]+\lim_{\theta\to 0} \E^t\left[ \frac{g(X_\tau, \hat{Q}^{\theta,\bar{c}, t}_{\tau^{\theta,\bar{c}, t}})-g(X_\tau, \hat{Q}^{\theta,\bar{c}, t}_\tau)}{\theta}\right]\\
&+\lim_{\theta\to 0} \E^t\left[ \frac{g(X_\tau, Q^{\theta,\bar{c}, t}_{\tau^{\theta,\bar{c}, t}})-g(X_\tau, Q^{\theta,\bar{c}, t}_{\tau^{\theta,\bar{c}, t}_{\min}})+g(X_\tau, Q_{\tau^{\theta,\bar{c}, t}_{\min}})-g(X_\tau, \hat{Q}^{\theta,\bar{c}, t}_{\tau^{\theta,\bar{c}, t}})+g(X_\tau, \hat{Q}^{\theta,\bar{c}, t}_\tau)-g(X_\tau, Q_\tau)}{\theta}\right].
\end{split}
\end{equation}
Reminding that when $\tau^{\theta,\bar{c}, t}_{\min}=\tau^{\theta,\bar{c}, t}$, then $\hat{Q}^{\theta,\bar{c}, t}_{\tau^{\theta,\bar{c}, t}}=Q_{\tau^{\theta,\bar{c}, t}}$ and $\hat{Q}^{\theta,\bar{c}, t}_{\tau}=Q_\tau$ and when $\tau^{\theta,\bar{c}, t}_{\min}=\tau$, then $\hat{Q}^{\theta,\bar{c}, t}_{\tau^{\theta,\bar{c}, t}}=Q^{\theta,\bar{c}, t}_{\tau^{\theta,\bar{c}, t}}$ and $\hat{Q}^{\theta,\bar{c}, t}_{\tau}=Q^{\theta,\bar{c}, t}_{\tau}$, then we have that the last element on the right-hand side of $\eqref{derivationepspos}$ is equal to $0$. The first element on the right-hand side of $\eqref{derivationepspos}$ is equal to $-\bar{g}(t,\bar{c},x,q)$ by its definition $\eqref{defmubarpos}$. We define $\tilde{g}$ for any $(x,q)\in\mathcal{O}, (x,q')\in\mathcal{O}$ as
\begin{equation*}
\tilde{g}(x,q,q'):=\begin{cases}
\frac{g(x,q)-g(x,q')}{q-q'} &\text{if } q\ne q',\\
\partial_q g(x,q') &\text{if } q=q'.
\end{cases}
\end{equation*}
From Assumption \ref{assumptions} we have that $\tilde{g}$ is bounded by $K(1+|x|)$. The second element on the right-hand side of $\eqref{derivationepspos}$ is equal to
\begin{equation}
\label{eqq545third}
\begin{split}
\lim_{\theta\to 0} \E^t&\left[ \tilde{g}\left( X_\tau, Q^{\theta,\bar{c}, t}_{\tau^{\theta,\bar{c}, t}_{\min}}, Q_{\tau^{\theta,\bar{c}, t}_{\min}} \right)\frac{Q^{\theta,\bar{c}, t}_{\tau^{\theta,\bar{c}, t}_{\min}}-Q_{\tau^{\theta,\bar{c}, t}_{\min}}}{\theta}\right]\\
&=\lim_{\theta\to 0} \E^t\left[ \tilde{g}\left( X_\tau, Q^{\theta,\bar{c}, t}_{\tau^{\theta,\bar{c}, t}_{\min}}, Q_{\tau^{\theta,\bar{c}, t}_{\min}} \right)\left(\frac{Q^{\theta,\bar{c}, t}_{\tau^{\theta,\bar{c}, t}_{\min}}-Q_{\tau^{\theta,\bar{c}, t}_{\min}}}{\theta}+\bar{c}-c_t\right)\right]\\
&\qquad-\lim_{\theta\to 0} \E^t\left[ (\bar{c}-c_t)\left(\tilde{g}\left( X_\tau, Q^{\theta,\bar{c}, t}_{\tau^{\theta,\bar{c}, t}_{\min}}, Q_{\tau^{\theta,\bar{c}, t}_{\min}} \right)-\partial_q g\left(X_\tau,Q_{\tau^{\theta,\bar{c}, t}_{\min}}\right)\right)\right]\\
&\qquad- (\bar{c}-c_t) \lim_{\theta\to 0}\E^t\left[ \partial_q g\left(X_\tau,Q_{\tau^{\theta,\bar{c}, t}_{\min}}\right)-\partial_q g\left(X_\tau,Q_\tau\right)\right]\\
&\qquad-(\bar{c}-c_t) \E^t\left[ \partial_q g\left(X_\tau,Q_\tau\right)\right].
\end{split}
\end{equation}
Using H\"older's inequality, boundedness of $\tilde{g}$ and $\eqref{lemmaonqpos}$, we get
\begin{equation*}
\begin{split}
\lim_{\theta\to 0} \E^t&\left[ \left|\tilde{g}\left( X_\tau, Q^{\theta,\bar{c}, t}_{\tau^{\theta,\bar{c}, t}_{\min}}, Q_{\tau^{\theta,\bar{c}, t}_{\min}} \right)\left(\frac{Q^{\theta,\bar{c}, t}_{\tau^{\theta,\bar{c}, t}_{\min}}-Q_{\tau^{\theta,\bar{c}, t}_{\min}}}{\theta}+\bar{c}-c_t\right)\right|\right]\\
&\le K\left(\E^t\left[ \left(1+\left| X_\tau\right|\right)^4\right]\right)^{\frac 14}\lim_{\theta\to 0} \left(\E^t\left[ \left|\frac{Q^{\theta,\bar{c}, t}_{\tau^{\theta,\bar{c}, t}_{\min}}-Q_{\tau^{\theta,\bar{c}, t}_{\min}}}{\theta}+\bar{c}-c_t\right|^{\frac{4}{3}}\right]\right)^{\frac 34}=0.
\end{split}
\end{equation*}
Here we used standard arguments of SDE theory, i.e. $\E^t\left[\sup_{r\in[t,T]}|X_r|^4\right]<\infty$. Moreover, using Lemma \ref{lemma4} together with definition of $\tilde{g}$, we get that
\begin{equation*}
\begin{split}
\lim_{\theta\to 0} &\E^t\left[|\bar{c}-c_t|\left| \tilde{g}\left( X_\tau, Q^{\theta,\bar{c}, t}_{\tau^{\theta,\bar{c}, t}_{\min}}, Q_{\tau^{\theta,\bar{c}, t}_{\min}} \right)-\partial_q g\left(X_{\tau}, Q_{\tau^{\theta,\bar{c}, t}_{\min}}\right)\right|\right]\\
&\qquad\le \frac{K}{2}|\bar{c}-c_t|\lim_{\theta\to 0} \E^t\left[\left| Q^{\theta,\bar{c}, t}_{\tau^{\theta,\bar{c}, t}_{\min}}-Q_{\tau^{\theta,\bar{c}, t}_{\min}} \right|\mathds{1}_{Q^{\theta,\bar{c}, t}_{\tau^{\theta,\bar{c}, t}_{\min}}\ne Q_{\tau^{\theta,\bar{c}, t}_{\min}}}\right]=0,
\end{split}
\end{equation*}
where in the last line we used $\eqref{relonlimdiffQ}$ in Lemma \ref{lemmaonlimdiffQ}. Moreover, using Lipschitz continuity of $\partial_q g$ Lemma $\eqref{lemmaconvrvtheta}$ and that either $\tau^{\theta,\bar{c}, t}_{\min}=\tau^{\theta,\bar{c}, t}$ or $\tau^{\theta,\bar{c}, t}_{\min}=\tau$, we get that
\begin{equation*}
\begin{split}
\lim_{\theta\to 0} \E^t\left[ \left|\partial_q g\left(X_\tau,Q_{\tau^{\theta,\bar{c}, t}_{\min}}\right)-\partial_q g\left(X_\tau,Q_\tau\right)\right|\right]&\le K\lim_{\theta\to 0} \E^t\left[ \left|Q_{\tau^{\theta,\bar{c}, t}}-Q_\tau\right|\mathds{1}_{\tau^{\theta,\bar{c}, t}_{\min}=\tau^{\theta,\bar{c}, t}}\right]=0.
\end{split}
\end{equation*}
Hence, merging the last three expressions above into $\eqref{eqq545third}$, we get
\begin{equation}
\label{eqq545finalthird}
\begin{split}
\lim_{\theta\to 0} \E^t\left[ \frac{g(X_\tau, Q^{\theta,\bar{c}, t}_{\tau^{\theta,\bar{c}, t}_{\min}})-g(X_\tau, Q_{\tau^{\theta,\bar{c}, t}_{\min}})}{\theta}\right]=-(\bar{c}-c_t) \E^t\left[ \partial_q g\left(X_\tau,Q_\tau\right)\right].
\end{split}
\end{equation}
The third element on the right-hand side of $\eqref{derivationepspos}$ is equal to
\begin{equation}
\label{eqq545last}
\begin{split}
&\lim_{\theta\to 0} \E^t\left[ \tilde{g}\left( X_\tau, \hat{Q}^{\theta,\bar{c}, t}_{\tau^{\theta,\bar{c}, t}}, \hat{Q}^{\theta,\bar{c}, t}_\tau \right)\frac{\hat{Q}^{\theta,\bar{c}, t}_{\tau^{\theta,\bar{c}, t}}-\hat{Q}^{\theta,\bar{c}, t}_\tau}{\theta}\right]\\
&=\lim_{\theta\to 0} \E^t\left[ \tilde{g}\left( X_\tau, \hat{Q}^{\theta,\bar{c}, t}_{\tau^{\theta,\bar{c}, t}}, \hat{Q}^{\theta,\bar{c}, t}_\tau \right)\left(\frac{\hat{Q}^{\theta,\bar{c}, t}_{\tau^{\theta,\bar{c}, t}}-\hat{Q}^{\theta,\bar{c}, t}_\tau}{\theta}-(\bar{c}-c_t)\mathds{1}_{\Lambda(t,\bar{c})}\right)\right]\\
&\quad+\lim_{\theta\to 0} \E^t\left[ \left(\tilde{g}\left( X_\tau, \hat{Q}^{\theta,\bar{c}, t}_{\tau^{\theta,\bar{c}, t}}, \hat{Q}^{\theta,\bar{c}, t}_\tau \right)-\partial_q g\left(X_\tau,\hat{Q}^{\theta,\bar{c}, t}_\tau\right)\right)(\bar{c}-c_t)\mathds{1}_{\Lambda(t,\bar{c})}\right]\\
&\quad+\lim_{\theta\to 0} \E^t\left[ \left(\partial_q g\left(X_\tau,\hat{Q}^{\theta,\bar{c}, t}_\tau\right)-\partial_q g\left(X_\tau,Q_\tau\right)\right)(\bar{c}-c_t)\mathds{1}_{\Lambda(t,\bar{c})}\right]\\
&\quad+(\bar{c}-c_t) \E^t\left[\partial_q g\left(X_\tau,Q_\tau\right)\mathds{1}_{\Lambda(t,\bar{c})}\right].
\end{split}
\end{equation}
Using H\"older's inequality, boundedness of $\tilde{g}$ and Lemma \ref{lemma5}, we get
\begin{equation*}
\begin{split}
\lim_{\theta\to 0} \E^t&\left[ \left|\tilde{g}\left( X_\tau, \hat{Q}^{\theta,\bar{c}, t}_{\tau^{\theta,\bar{c}, t}}, \hat{Q}^{\theta,\bar{c}, t}_\tau \right)\left(\frac{\hat{Q}^{\theta,\bar{c}, t}_{\tau^{\theta,\bar{c}, t}}-\hat{Q}^{\theta,\bar{c}, t}_\tau}{\theta}-(\bar{c}-c_t)\mathds{1}_{\Lambda(t,\bar{c})}\right)\right|\right]\\
&\le K\left(\E^t\left[ \left(1+\left| X_\tau\right|\right)^4\right]\right)^{\frac 14}\lim_{\theta\to 0} \left(\E^t\left[ \left|\frac{\hat{Q}^{\theta,\bar{c}, t}_{\tau^{\theta,\bar{c}, t}}-\hat{Q}^{\theta,\bar{c}, t}_\tau}{\theta}-(\bar{c}-c_t)\mathds{1}_{\Lambda(t,\bar{c})}\right|^{\frac{4}{3}}\right]\right)^{\frac 34}=0.
\end{split}
\end{equation*}
Here we used standard arguments of SDE theory, i.e. $\E^t\left[\sup_{r\in[0,T]}|X_r|^4\right]<\infty$. Moreover, using Lemma \ref{lemma4} together with definition of $\tilde{g}$, we get that
\begin{equation*}
\begin{split}
\lim_{\theta\to 0} &\E^t\left[|\bar{c}-c_t|\left| \tilde{g}\left( X_\tau, \hat{Q}^{\theta,\bar{c}, t}_{\tau^{\theta,\bar{c}, t}}, \hat{Q}^{\theta,\bar{c}, t}_\tau \right)-\partial_q g\left(X_{\tau}, \hat{Q}^{\theta,\bar{c}, t}_\tau\right)\right|\mathds{1}_{\Lambda(t,\bar{c})}\right]\\
&\le \frac{K}{2}|\bar{c}-c_t|\lim_{\theta\to 0} \E^t\left[\left| \hat{Q}^{\theta,\bar{c}, t}_{\tau^{\theta,\bar{c}, t}}-\hat{Q}^{\theta,\bar{c}, t}_\tau \right|\mathds{1}_{\hat{Q}^{\theta,\bar{c}, t}_{\tau^{\theta,\bar{c}, t}}\ne\hat{Q}^{\theta,\bar{c}, t}_\tau}\right]=0,
\end{split}
\end{equation*}
where in the last line we used $\eqref{eqconvQhat}$ in Lemma \ref{lemmaconvQhat}. Moreover, using Lipschitz continuity of $\partial_q g$ in Assumption \ref{assumptions}, we have that
\begin{equation*}
\begin{split}
\lim_{\theta\to 0} \E^t&\left[ \left|(\bar{c}-c_t)\mathds{1}_{\Lambda(t,\bar{c})}\left(\partial_q g\left(X_\tau,\hat{Q}^{\theta,\bar{c}, t}_\tau\right)-\partial_q g\left(X_\tau,Q_\tau\right)\right)\right|\right]\\
&\le K|\bar{c}-c_t|\lim_{\theta\to 0}\E\left[ \left|\hat{Q}^{\theta,\bar{c}, t}_\tau-Q_\tau\right|\right]=0,
\end{split}
\end{equation*}
where in the last equality we have used $\eqref{eqconvQhatsingle}$ in Lemma \ref{lemmaconvQhat}. Hence, merging the last three expressions above into $\eqref{eqq545last}$, we get
\begin{equation*}
\begin{split}
\lim_{\theta\to 0} \E^t&\left[ \frac{g(X_\tau, \hat{Q}^{\theta,\bar{c}, t}_{\tau^{\theta,\bar{c}, t}})-g(X_\tau, \hat{Q}^{\theta,\bar{c}, t}_\tau)}{\theta}\right]=(\bar{c}-c_t) \E^t\left[\partial_q g\left(X_\tau,Q_\tau\right)\mathds{1}_{\Lambda(t,\bar{c})}\right].
\end{split}
\end{equation*}
Combining $\eqref{defmubarpos}$, $\eqref{eqq545finalthird}$ and the above expression into $\eqref{derivationepspos}$, we conclude that the first line of the right-hand side of $\eqref{derivationepsposfirstpart}$ is equal to
\begin{equation}
\label{eqq353}
\begin{split}
\lim_{\theta\to 0} &\E^t\left[ \frac{g(X_{\tau^{\theta,\bar{c}, t}}, Q^{\theta,\bar{c}, t}_{\tau^{\theta,\bar{c}, t}})-g(X_\tau, Q_\tau)}{\theta}\right]=-\bar{g}(t,\bar{c},x,q)-\E^t\left[(\bar{c}-c_t) \partial_q g(X_\tau, Q_\tau)\right]\\
&\qquad+(\bar{c}-c_t)\E^t\left[\partial_q g(X_\tau, Q_\tau) \mathds{1}_{\Lambda(t,\bar{c})}\right].
\end{split}
\end{equation}
The second and third lines of right-hand side of $\eqref{derivationepsposfirstpart}$ can be written as
\begin{equation}
\label{secondlinec27}
\begin{split}
&\lim_{\theta\to 0} \E^t\left[\frac{1}{\theta}\int_{t}^{\tau^{\theta,\bar{c}, t}_{\min}}\left(f(r, c^{\theta,\bar{c}, t}_r, X_r,Q^{\theta,\bar{c}, t}_r)-f(r, c_r, X_r,Q_r)\right) dr\right]\\
&\qquad\qquad-\lim_{\theta\to 0} \E^t\left[\frac{\sign(\tau-\tau^{\theta,\bar{c}, t})}{\theta}\int_{\tau^{\theta,\bar{c}, t}_{\min}}^{\tau^{\theta,\bar{c}, t}_{\max}}f(r, \hat{c}^{\theta,\bar{c}, t}_r, X_r,\hat{Q}^{\theta,\bar{c}, t}_r) dr\right]\\
&\qquad =\lim_{\theta\to 0} \E^t\left[\frac{1}{\theta}\int_{t}^{\tau^{\theta,\bar{c}, t}_{\min}}\left(f(r, c^{\theta,\bar{c}, t}_r, X_r,Q^{\theta,\bar{c}, t}_r)-f(r, c_r, X_r,Q_r)\right) dr-\xi_{\tau^{\theta,\bar{c}, t}_{\min}}\right]+\lim_{\theta\to 0} \E^t\left[\xi_{\tau^{\theta,\bar{c}, t}_{\min}}\right]\\
&\qquad\qquad-\lim_{\theta\to 0} \E^t\left[ \frac{\sign(\tau-\tau^{\theta,\bar{c}, t})}{\theta}\int_{\tau^{\theta,\bar{c}, t}_{\min}}^{\tau^{\theta,\bar{c}, t}_{\max}}f(r, \hat{c}^{\theta,\bar{c}, t}_r, X_r,\hat{Q}^{\theta,\bar{c}, t}_r) dr\right].
\end{split}
\end{equation}
Using Lemma \ref{lemma3onf}, we have that
\begin{equation*}
\lim_{\theta\to 0} \E^t\left[\frac{1}{\theta}\int_{t}^{\tau^{\theta,\bar{c}, t}_{\min}}\left(f(r, c^{\theta,\bar{c}, t}_r, X_r,Q^{\theta,\bar{c}, t}_r)-f(r, c_r, X_r,Q_r)\right) dr-\xi_{\tau^{\theta,\bar{c}, t}_{\min}}\right]=0.
\end{equation*}
Using Lemma \ref{lemmaconvrvtheta} ad reminding that either $\tau^{\theta,\bar{c}, t}_{\min}=\tau^{\theta,\bar{c}, t}$ or $\tau^{\theta,\bar{c}, t}_{\min}=\tau$, we have that
\begin{equation*}
\begin{split}
\lim_{\theta\to 0} \E^t\left[\xi_{\tau^{\theta,\bar{c}, t}_{\min}}\right]&=\lim_{\theta\to 0} \E^t\left[\left(\xi_{\tau^{\theta,\bar{c}, t}}-\xi_{\tau}\right)\mathds{1}_{\tau^{\theta,\bar{c}, t}_{\min}=\tau^{\theta,\bar{c}, t}}\right]+ \E^t\left[\xi_{\tau}\right]=\E^t\left[\xi_{\tau}\right].
\end{split}
\end{equation*}
Using $\eqref{deffbarpos}$, the third limit on the right-hand side of $\eqref{secondlinec27}$ converges to $\bar{f}(t,\bar{c},x,q)$. Combining the above two expressions and $\eqref{deffbarpos}$ into $\eqref{secondlinec27}$, we get that
\begin{equation*}
\begin{split}
&\lim_{\theta\to 0} \E^t\left[\frac{1}{\theta}\int_{t}^{\tau^{\theta,\bar{c}, t}_{\min}}\left(f(r, c^{\theta,\bar{c}, t}_r, X_r,Q^{\theta,\bar{c}, t}_r)-f(r, c_r, X_r,Q_r)\right) dr\right]\\
&\qquad\qquad-\lim_{\theta\to 0} \E^t\left[\frac{\sign(\tau-\tau^{\theta,\bar{c}, t})}{\theta}\int_{\tau^{\theta,\bar{c}, t}_{\min}}^{\tau^{\theta,\bar{c}, t}_{\max}}f(r, \hat{c}^{\theta,\bar{c}, t}_r, X_r,\hat{Q}^{\theta,\bar{c}, t}_r) dr\right]=\E^t\left[\xi_\tau\right]-\bar{f}(t,\bar{c},x,q).
\end{split}
\end{equation*}
Then, merging $\eqref{derivationepsposfirstpart}$  with $\eqref{eqq353}$ and the above expression, also noting  $X_t=x$ and $Q_t=q$, we get
\begin{equation}
\label{eqq1}
\begin{split}
\lim_{\theta\to 0} \frac{v^{c^{\theta,\bar{c}, t}}(t, x,q)-v^c(t,x,q)}{\theta}&= \E^t\left[-(\bar{c}-c_t) \partial_q g(X_\tau, Q_\tau) +\xi_\tau\right]+\mathcal{G}(t,\bar{c},X_t,Q_t),
\end{split}
\end{equation}
where $\mathcal{G}(t,\bar{c},x,q)$ is defined in (\ref{defGcal}). 

However, from $\eqref{defY}$ and $\eqref{defxi}$,  we have 
\begin{equation*}
\begin{split}
\E^t&\left[ -(\bar{c}-c_t)\partial_q g(X_\tau, Q_\tau) +\xi_\tau\right]= \E^t\left[-(\bar{c}-c_t) Y_\tau +\xi_\tau\right]\\
&= \E^t\left[-(\bar{c}-c_t) Y_t -(\bar{c}-c_t)\int_t^\tau dY_r+\xi_t+\int_t^\tau d\xi_r\right] \\
&=\E^t\left[-(\bar{c}-c_t) Y_t + f(t,  \bar{c}, X_{t}, Q_{t})-f(t, c_{t}, X_{t}, Q_{t})\right].
\end{split}
\end{equation*}
In the last equality we have used the Optional Stopping Theorem, which ensures that $\int_t^\cdot  Z_r dW_r$ is a martingale, whose conditional expectation is $0$. Substituting the above expression into $\eqref{eqq1}$ and using the optimality condition  $\eqref{negativeJ}$, we get that for any $\bar{c}\ge 0$ and $t\in[0,\tau)$
\begin{equation*}
\begin{split}
&\E^t\left[-(\bar{c}-c_t) Y_t + f(t,  \bar{c}, X_{t}, Q_{t})-f(t, c_{t}, X_{t}, Q_{t})\right]+\mathcal{G}(t,\bar{c},X_t,Q_t)\le 0.
\end{split}
\end{equation*}
Since the argument of the first conditional expectation is $\mathcal{F}^t$-measurable, using $\mathcal{H}$ in $\eqref{defHcal}$, we get the Hamiltonian condition $\eqref{SMP}$. This concludes the proof of Theorem \ref{SMPtheorem}.

\section{Conclusions}
\label{sectionconclusion}
In this paper we have proved a new SMP (Theorem \ref{SMPtheorem}) for an optimal liquidation problem with control-dependent terminal time, which is markedly different in the Hamiltonian condition from that of the standard SMP. We have given a simple example to show that the optimal solution satisfies the SMP in  Theorem \ref{SMPtheorem} but not the standard SMP in the literature. This is only the first step in the direction of SMP for control-dependent stopping time problems and there remain many open questions to be answered, for example,  
existence of pointwise limits $\eqref{defmubarpos}$ and $\eqref{deffbarpos}$, sufficient SMP for optimality,  a jump diffusion control-dependent model for $\X$ process, and applications to concrete financial scenarios. We leave these and other questions for future research.

\bigskip\noindent
{\bf Acknowledgment}. The authors are very grateful to  two anonymous reviewers whose constructive comments and suggestions have helped to improve the paper of the previous three versions.

\bibliographystyle{abbrv}

\end{document}